\documentclass[11pt]{article}
\usepackage{amsfonts}
\usepackage{amsmath,amssymb}
\usepackage{amsthm}
\usepackage[mathscr]{euscript}
\usepackage{mathrsfs}
\usepackage{indentfirst}
\usepackage{color}
\usepackage{accents}
\usepackage{enumerate}

\newtheorem{theorem}{Theorem}[section]
\newtheorem{lemma}{Lemma}[section]

\newtheorem{proposition}{Proposition}[section]
\numberwithin{equation}{section}

\newcommand{\lbl}[1]{\label{#1}}
\allowdisplaybreaks

\newcommand{\be}{\begin{equation}}
\newcommand{\ee}{\end{equation}}
\newcommand\bes{\begin{eqnarray}} \newcommand\ees{\end{eqnarray}}
\newcommand{\bess}{\begin{eqnarray*}}
\newcommand{\eess}{\end{eqnarray*}}
\newcommand{\bbbb}{\left\{\begin{aligned}}
\newcommand{\nnnn}{\end{aligned}\right.}
\newcommand{\bea}{\begin{align*}}
\newcommand{\eea}{\end{align*}}

\newcommand\ep{\varepsilon}

\newcommand\kk{\left}
\newcommand\rr{\right}
\newcommand\dd{\displaystyle}

\newcommand\df{\dd\frac}

\newcommand\dx{{\rm d}x}
\newcommand\dy{{\rm d}y}

\newcommand\yy{\infty}
\newcommand\qq{\eqref}

\newcommand\ol{\overline}
\newcommand\ud{\underline}

\begin{document}\thispagestyle{empty}
\setlength{\baselineskip}{16pt}
\begin{center}
 {\LARGE\bf Dynamics of an epidemic model}\\[2mm]
 {\LARGE\bf with nonlocal diffusion and a free boundary\footnote{This work was supported by NSFC Grants
 12171120, 12301247}}\\[4mm]
{\Large Lei Li}\\[0.5mm]
{College of Science, Henan University of Technology, Zhengzhou, 450001, China}\\[2.5mm]
  {\Large Mingxin Wang\footnote{Corresponding author. {\sl E-mail}: mxwang@hpu.edu.cn}}\\[0.5mm]
 {School of Mathematics and Information Science, Henan Polytechnic University, Jiaozuo, 454003, China}
\end{center}

\date{\today}

\begin{quote}
\noindent{\bf Abstract.} An epidemic model, where the dispersal is approximated by nonlocal diffusion operator and spatial domain has one fixed boundary and one free boundary, is considered in this paper. Firstly, using some elementary analysis instead of variational characterization, we show the existence and asymptotic behaviors of the principal eigenvalue of a cooperative system which can be used to characterize more epidemic models, not just ours. Then we study the existence, uniqueness and stability of a related steady state problem. Finally, we obtain a rather complete understanding for long time behaviors, spreading-vanishing dichotomy, criteria for spreading and vanishing, and spreading speed. Particularly, we prove that the asymptotic spreading speed of solution component $(u,v)$ is equal to the spreading speed of free boundary which is finite if and only if a threshold condition holds for kernel functions.


\end{quote}

\section{Introduction}\pagestyle{myheadings}
\renewcommand{\thethm}{\Alph{thm}}
{\setlength\arraycolsep{2pt}

To model the spread of an oral-faecal transmitted epidemic, Hsu and Yang \cite{HY} proposed the following PDE system
\bes\label{1.1}
\left\{\!\begin{aligned}
&u_{t}=d_1\Delta u-au+H(v), & &t>0,~x\in\mathbb{R},\\
&v_{t}=d_2\Delta v-bv+G(u), & &t>0,~x\in\mathbb{R},
\end{aligned}\right.
 \ees
which is used to model the oral-faecal transmitted epidemic, where $H(v)$ and $G(u)$ satisfy
\begin{enumerate}
\item[{\bf(H)}]\; $H,G\in C^2([0,\yy))$, $H(0)=G(0)=0$, $H'(z),G'(z)>0$ in $[0,\yy)$, $H''(z), G''(z)<0$ in $(0,\yy)$, and $G(H(\hat z)/a)<b\hat{z}$ for some $\hat{z}>0$.
 \end{enumerate}
An example for such $H$ and $G$ is $H(z)=\alpha z/(1+z)$ and $G(z)=\beta \ln(z+1)$ with $\alpha,\beta>0$. In model \eqref{1.1}, $u(t,x)$ and $v(t,x)$ stand for the spatial
concentrations of the bacteria and the infective human population, respectively, at time $t$ and location $x$ in the one dimensional habitat; $-au$ represents the natural death rate of the bacterial population and $H(v)$ denotes the contribution of the infective human to the growth rate of the bacteria; $-bv$ is the fatality rate of the infective human population and $G(u)$ is the infection rate of human population; $d_1$ and $d_2$, respectively, stand for the diffusion rate of bacteria and infective human. Define
  \bes
  \mathcal{R}_0=\frac{H'(0)G'(0)}{ab}.
  \lbl{x.1}\ees
When $\mathcal{R}_0>1$, the authors proved that there exists a $c_*>0$ such that \eqref{1.1} has a positive monotone travelling wave solution if and only if $c\ge c_*$. Moreover, dynamics of the corresponding ODE system with positive initial value is govern by $\mathcal{R}_0$. More precisely, when $\mathcal{R}_0<1$, $(0,0)$ is globally asymptotically stable; while when $\mathcal{R}_0>1$, there exists a unique positive equilibrium $(U,V)$ which is uniquely given by
  \bes\label{1.2}
  aU=H(V), ~ ~ bV=G(U),\ees
and is globally asymptotically stable.

If $H(v)=cv$, then system \eqref{1.1} reduces to
\bes\label{1.3}
u_{t}=d_1\Delta u-au+cv, ~ ~ v_{t}=d_2\Delta v-bv+G(u), t>0, ~ x\in\mathbb{R}
 \ees
whose corresponding ODE system was proposed in \cite{CP} to describe the 1973 cholera epidemic spread in the European Mediterranean regions. Here $G$ satisfies that $G\in C^2([0,\yy))$, $G(0)=0<G'(u)$ in $[0,\yy)$, $G(u)/u$ is strictly decreasing in $(0,\yy)$ and $\lim\limits_{u\to\yy}{G(u)}/{u}<{ab}/{c}$. From {\bf(H)} and the assumption on $G$ of \eqref{1.3}, it can be learned that both \eqref{1.1} and \eqref{1.3} are monostable cooperative systems, which have been extensively used to describe the spread of epidemic, such as cholera, typhoid fever and West Nile virus, etc.  When modeling epidemic, an important issue is to know where the spreading frontier of epidemic is located, which naturally motivates us to discuss the systems, such as \eqref{1.1} and \eqref{1.3}, on the domain whose boundary is unknown and varies over time, instead of the fixed boundary domain or the whole space.

As a pioneering work where free boundary condition is incorporated into the model arising from ecology, Du and Lin \cite{DL} proposed the following problem
 \bes\label{1.4}
 \left\{\!\begin{aligned}
&u_{t}=d\Delta u+u(a-bu), & &t>0,~x\in(g(t),h(t)),\\
&u(t,g(t))=u(t,h(t))=0, & &t>0,\\
&g'(t)=-\mu u_x(t,g(t)),~ ~ h'(t)=-\mu u_x(t,h(t)), & & t>0,\\
&-g(0)=h(0)=h_0>0, ~ u(0,x)=\hat{u}_0(x), & & x\in[-h_0,h_0],
\end{aligned}\right.
 \ees
where $\hat{u}_0(x)$ is assumed to satisfy $\hat{u}_0(x)\in C^2([-h_0,h_0])$, $\hat{u}_0(\pm h_0)=0<\hat{u}_0(x)$ in $(-h_0,h_0)$. The free boundary condition $g'(t)=-\mu u_x(t,g(t))$ and $h'(t)=-\mu u_x(t,h(t))$ is usually referred to as the Stefan boundary condition. Du and Lin found that the dynamics of \eqref{1.4} is govern by a spreading-vanishing dichotomy, a new spreading phenomena resulting from reaction-diffusion model. Besides, when spreading happens, the speed was also obtained by analyzing a semi-wave problem.

As we can see, the dispersal in the above models is approximated by random diffusion $\Delta u$. Recently, it has been increasingly recognized that nonlocal diffusion is better to describe the spatial dispersal, since such diffusion operator can capture local as well as long-distance dispersal. A commonly used nonlocal diffusion operator takes the form of
\bes\label{1.6}
d\int_{\mathbb{R}^N}J(|x-y|)u(t,y)\dy-du,\ees
where $J$ is the kernel function and $d$ is the diffusion coefficient. A biological interpretation of \eqref{1.6} and its properties can be seen from, for example, \cite{AMRT,KLS,BCV,ZhangW24}.
Using operator \eqref{1.6} or its variation to model the spreading phenomenon from ecology and epidemiology has attracted much attention, and many related works have emerged over past decades. An important difference, compared to the classical reaction-diffusion equations, is that spreading speed may be infinite, known as accelerated spreading, if $J$ violates a so-called ``thin tailed'' condition. For example, please see \cite{Ya, Gar, AC, XLR}.

Replacing random diffusion $\Delta u$ in \eqref{1.4} with nonlocal diffusion operator \eqref{1.6}, Cao et al \cite{CDLL} and Cort{\'a}zar et al \cite{CQW} independently considered the following problem
\bes\left\{\begin{aligned}\label{1.7}
&u_t=d\int_{g(t)}^{h(t)}\!\!J(x-y)u(t,y)\dy-du+f(u), & &t>0,~x\in(g(t),h(t)),\\[1mm]
&u(t,x)=0,& &t>0, ~ x\notin(g(t),h(t))\\
&h'(t)=\mu\int_{g(t)}^{h(t)}\!\!\int_{h(t)}^{\infty}
J(x-y)u(t,x)\dy\dx,& &t>0,\\[1mm]
&g'(t)=-\mu\int_{g(t)}^{h(t)}\!\!\int_{-\infty}^{g(t)}
J(x-y)u(t,x)\dy\dx,& &t>0,\\[1mm]
&h(0)=-g(0)=h_0>0,\;\; u(0,x)=u_0(x),& &|x|\le h_0,
 \end{aligned}\right.
 \ees
 where kernel $J$ satisfies
 \begin{enumerate}
\item[{\bf(J)}]$J\in C(\mathbb{R})\cap L^{\yy}(\mathbb{R})$, $J(x)\ge0$, $J(0)>0$, $J$ is even, $\int_{\mathbb{R}}J(x)\dx=1$,
 \end{enumerate}
and $u_0\in C([-h_0,h_0]$, $u_0(\pm h_0)=0<u_0(x)$ in $(-h_0,h_0)$.
The nonlinear term $f$ is of the Fisher-KPP type in \cite{CDLL} and $f\equiv0$ in \cite{CQW}.  The authors in \cite{CDLL} showed that similar to \eqref{1.4}, the dynamics of \eqref{1.7} is also govern by a spreading-vanishing dichotomy. However, when spreading occurs, it was proved in \cite{DLZ} that the spreading speed of \eqref{1.7} is finite if and only if $\int_0^{\yy}xJ(x)\dx<\yy$, which is much different from \eqref{1.4} since the spreading speed of \eqref{1.4} is always finite. In addition, there are other developments on research of \eqref{1.7} along different directions. Please see a series of works of Du and Ni \cite{DN2,DN3,DN4,DN5} for spreading speed in homogeneous environment and \cite{ZLZ} for the case in periodic environment. Particularly, the following variant of \eqref{1.7} was proposed by Li et al \cite{LLW}
\bes\label{1.8}
\left\{\begin{aligned}
&u_t=d\dd\int_0^{h(t)}J(x-y)u(t,y)\dy-dj(x)u+f(u), && t>0,~0\le x<h(t),\\
&u(t,h(t))=0, && t>0,\\
&h'(t)=\mu\dd\int_0^{h(t)}\int_{h(t)}^{\infty}
J(x-y)u(t,x)\dy\dx, && t>0,\\
&h(0)=h_0,\;\; u(0,x)=u_0(x), &&0\le x\le h_0,
\end{aligned}\right.
 \ees
 where $J$ and $f$ satisfy the same conditions with \eqref{1.7}, and $j(x)=\int_0^{\yy}J(x-y)\dy$; $u_0$ meets with
  \begin{enumerate}
\item[{\bf(I)}]$u_0\in C([0,h_0]$, $u_0(h_0)=0<u_0(x)$ in $[0,h_0)$.
 \end{enumerate}
This model is derived from the assumption that the species will never jump to the area $(-\yy,0)$ which is similar to the usual homogeneous Neumann boundary condition imposed at $x=0$.

It is well known that if further $\hat{u}_0(x)$ is even, then problem \eqref{1.4} can reduce to the model \cite[(1.1)]{DL} where spatial domain has one free boundary and one fixed boundary. Hence it is natural to think whether \eqref{1.7} and \eqref{1.8} are equivalent in some sense. We shall show that \qq{1.8} cannot be transformed into \qq{1.7} in the appendix (cf. Theorem \ref{ta.1}).

Nonlocal diffusion systems composed of \eqref{1.7} have been widely utilized to model the propagation of epidemic or species in epidemiology or ecology over past decades. Please refer to, for example, \cite{ DWZ22, LLW24, LW24} for the competition, prey-predator and mutualist models, \cite{ZZLD,ZLD,DLNZ,WD1,WD2,DNW} for related problems of \eqref{1.3}, \cite{DN8,PLL} for West Nile virus, \cite{ YYW23} for SIR model, and \cite{ZhangW23} for competition model with seasonal succession. Very recently, Nguyen and Vo \cite{NV} studied the following problem
\bes\left\{\!\begin{array}{ll}\label{1.9}
u_t=d_1\!\dd\int_{g(t)}^{h(t)}J_1(x-y)u(t,y)\dy-d_1u-au+H(v), \hspace{2mm}t>0,~x\in(g(t),h(t)),\\[4mm]
v_t=d_2\!\dd\int_{g(t)}^{h(t)}J_2(x-y)v(t,y)\dy-d_2v-bv+G(u), \hspace{5mm}t>0,~x\in(g(t),h(t)),\\[4mm]
u(t,g(t))=v(t,h(t))=0, \hspace{47mm}t>0,\\[2mm]
g'(t)=-\dd\int_{g(t)}^{h(t)}\!\int_{-\yy}^{g(t)}\big[\mu
J_1(x-y)u(t,x)+\mu\rho J_2(x-y)v(t,x)\big]\dy\dx, \;\;t>0,\\[5mm]
h'(t)=\dd\int_{g(t)}^{h(t)}\!\int_{h(t)}^{\yy}\big[\mu
J_1(x-y)u(t,x)+\mu\rho J_2(x-y)v(t,x)\big]\dy\dx,\quad\;\;\; t>0,\\[4mm]
-g(0)=h(0)=h_0>0,\; u(0,x)=u_0(x),\;v(0,x)=v_0(x),\hspace{5mm}|x|\le h_0,
 \end{array}\right.
 \ees
where $H$ and $G$ satisfy the condition {\bf(H)}. The authors obtained the well-posedeness, spreading-vanishing dichotomy as well as criteria for spreading and vanishing. Especially, for the self-adjoint case, they proved the existence and variational characteristic of a principal eigenvalue by Lax-Milgram's theorem, and further got its asymptotic behaviors by using variational characteristic.

Inspired by the above works, in this paper we shall investigate the following problem
  \bes\left\{\begin{aligned}\label{1.10}
&u_t=d_1\int_0^{h(t)}\!J_1(x-y)u(t,y)\dy-d_1j_1(x)u-au+H(v), \hspace{2mm} t>0, ~ x\in[0,h(t)),\\
&v_t=d_2\int_0^{h(t)}\!J_2(x-y)v(t,y)\dy-d_2j_2(x)v-bv+G(u), \hspace{3mm} t>0, ~ x\in[0,h(t)),\\
&u(t,h(t))=v(t,h(t))=0, \hspace{53mm} t>0,\\
&h'(t)=\int_0^{h(t)}\!\int_{h(t)}^{\yy}\big[\mu_1 J_1(x-y)u(t,x)
+\mu_2 J_2(x-y)v(t,x)\big]\dy\dx, \;\; t>0,\\
&h(0)=h_0>0,~ u(0,x)=u_0(x), ~ v(0,x)=v_0(x),\; ~ x\in[0,h_0],\\
 \end{aligned}\right.
 \ees
where all parameters are positive, $J_i$ satisfies the condition {\bf(J)}, and $j_i(x)=\int_0^{\yy}J_i(x-y)\dy$ for $i=1,2$.
Condition {\bf (I)} holds for $u_0$ and $v_0$. In this paper we assume that $H$ and $G$ satisfy  the following condition {\bf(H1)}, which is weaker than {\bf(H)},\vspace{-1.5mm}
   \begin{enumerate}
\item[{\bf(H1)}]\; $H,G\in C^1([0,\yy))$, $H(0)=G(0)=0$, $H'(z),G'(z)>0$ in $[0,\yy)$, ${H(z)}/z$ is decreasing in $z>0$, and ${G(z)}/z$ is strictly decreasing in $z>0$, and $G(H(\hat z)/a)<b\hat{z}$ for some $\hat{z}>0$.\vspace{-1.5mm}
 \end{enumerate}
This condition allows $H(v)=cv$, but condition {\bf(H)} does not include this case. Moreover, under the condition {\bf(H1)}, positive equilibrium $(U,V)$ also exists uniquely if $\mathcal{R}_0>1$.  Throughout this paper, we always assume that {\bf(H1)}, {\bf(J)} and {\bf(I)} hold.

By using similar methods as in \cite{DN8,NV} we can prove that
the problem \eqref{1.10} has a unique global solution $(u,v,h)$. Moreover, $(u,v)\in [C([0,\yy)\times[0,h(t)])]^2$, $h\in C^1([0,\yy))$, $0< u(t,x)\le M_1$, $0< v(t,x)\le M_2$ in $[0,\yy)\times[0,h(t))$ with some $M_1,M_2>0$, and $h'(t)>0$ for all $t\ge 0$. Thus $h_{\yy}:=\dd\lim_{t\to\yy}h(t)\in(h_0,\yy]$ is well defined. If $h_{\yy}<\yy$, we call {\it vanishing}; otherwise we call {\it spreading}.

In order to know as much as possible about the dynamics of \eqref{1.10}, in Section \ref{s2} we investigate the eigenvalue problem $\mathcal{L}[\varphi]=\lambda\varphi$ where operator $\mathcal{L}$ is defined by \eqref{2.1}. The existence of principal eigenvalue is obtained by using the arguments in \cite{SWZ}. When operator $\mathcal{L}$ is self-adjoint, we also get the related variational characteristic which is only used to show the monotonicity of principal eigenvalue on diffusion coefficient. More importantly, a rather complete understanding for the asymptotic behaviors about spatial domain and diffusion coefficients, which is crucial for studying the criteria for spreading and vanishing of \eqref{1.10}, is derived by a series of elementary analysis without assuming that $\mathcal{L}$ is self-adjoint.

With the help of principal eigenvalue, in Section \ref{s3} we first investigate the steady state problem associated to \eqref{1.10}, and then prove that the dynamics of evolutionary problem is determined completely by the sign of principal eigenvalue. Especially, when the principal eigenvalue is non-positive, it will be proved that $(0,0)$ is exponentially (principal eigenvalue is negative) or algebraically (principal eigenvalue is zero) stable.

In Section \ref{s4}, we establish the spreading-vanishing dichotomy, and give the long time behaviors of solution component $(u,v)$ and a rather complete description of criteria for spreading and vanishing by using the conclusions obtained in Sections 2 and 3.

When spreading happens, spreading speed is considered in Section \ref{s5}. We prove that the asymptotic spreading speed of solution component $(u,v)$ is equal to the spreading speed of free boundary which is finite if and only if a threshold condition holds for kernel functions.

Section 6 involves a discussion on the relations of \eqref{1.7} and \eqref{1.8}.}

Before ending the introduction, we emphasize the difference between \qq{1.9} and \qq{1.10}. Firstly, there is only one free boundary in \eqref{1.10} and no agents cross the fixed boundary $x=0$, which implies that agents can only expand their habitat to right side, while \eqref{1.9} allows agents to expand to both sides. Secondly, problem \eqref{1.9} is spatially homogeneous while problem \eqref{1.10} is spatially non-homogeneous, and \qq{1.10} cannot be transformed into \qq{1.9} by Theorem \ref{ta.1}. Thirdly, the eigenvalue problem corresponding to problem \qq{1.9} has constant coefficients, and its principal eigenvalue has shift  invariance, i.e., the principal eigenvalue defined on the interval $(l_1, l_2)$  depends only on the length $l_2-l_1$ but not on the position of $(l_1, l_2)$; whereas problem \qq{1.10} does not have such a good property.

\section{An eigenvalue problem associated to \eqref{1.10}}\lbl{s2}
{\setlength\arraycolsep{2pt}

For later discussion about the dynamics of \eqref{1.10}, in this section, we first study an eigenvalue problem of a cooperative system with nonlocal diffusion. In particular, without assuming that the operator is self-adjoint, we obtain a rather complete understanding of asymptotic behaviors of the principal eigenvalue which is expected to be useful in other cooperative nonlocal diffusion problems.

For any $a_{11},a_{22}\in\mathbb{R}$, $l>0$, $a_{12},a_{21}>0$, $d_1,d_2\ge0$ and $d_1+d_2>0$, we define the following nonlocal operator
 \bes\label{2.1}
 \mathcal{L}[\phi](x):=\mathcal{P}[\phi](x)+H(x)\phi(x), ~ ~x\in[0,l],\ees
where $\phi=(\phi_1,\phi_2)^T$,
 \[\mathcal{P}[\phi](x)=\begin{pmatrix}
 d_1\int_0^lJ_1(x-y)\phi_1(y)\dy \\
 d_2\int_0^lJ_2(x-y)\phi_2(y)\dy
 \end{pmatrix},~ ~ H(x)=\begin{pmatrix}
  -d_1j_1(x)+a_{11} & a_{12} \\
   a_{21} & -d_2j_2(x)+a_{22}
  \end{pmatrix}.\]}
Since we assume $d_1+d_2>0$ and $d_i\ge0$ for $i=1,2$, our results below can be used to handle some degenerate cooperative systems, such as \cite{ZZLD,LSch}. For clarity, we make some notations as follows.{\setlength\arraycolsep{2pt}
 \bess
E&=&[L^2([0,l])]^2, ~ \langle\phi,\psi\rangle=\sum_{i=1}^2\int_0^l\phi_i(x)\psi_i(x)\dx, ~  \|\phi\|_2=\sqrt{\langle\phi,\phi\rangle}, ~ X=[C([0,l])]^2,\\
X^+&=&\{\phi\in X: \phi_1\ge0, \phi_2\ge0 ~ {\rm in ~ }[0,l]\}, ~ X^{++}=\{\phi\in X: \phi_1>0, \phi_2>0 ~ {\rm in ~ }[0,l]\}.
 \eess

Now we are in the position to study the eigenvalue problem $\mathcal{L}[\phi]=\lambda\phi$.
It is well known that $\lambda$ is a principal eigenvalue if it is simple and its corresponding eigenfunction $\phi$ belongs to $X^{++}$. In the following, we first give the existence and some properties for principal eigenvalue of \eqref{2.1} by using the results in \cite{SWZ} whose proofs are inspired by the arguments in \cite{LCW}. When $\mathcal{L}$ is self-adjoint, we get a variational characteristic by following lines in the proofs of \cite[Theorem 2.3]{NV} and \cite[Theorem 3.1]{HMMV}, but our arguments are more concise than them.

\begin{proposition}\label{p2.1} Let $\mathcal{L}$ be defined as above. Then the following statements are valid.\vspace{-1.5mm}
 \begin{enumerate}[$(1)$]
 \item $\lambda_{p}$ is an eigenvalue of operator $\mathcal{L}$ with a corresponding eigenfunction $\phi_p\in X^{++}$, where
 \[\lambda_p=\inf\{\lambda\in\mathbb{R}: \mathcal{L}[\phi](x)\le\lambda\phi(x) ~ {\rm in } ~ [0,l] { \rm ~ for ~ some ~ } \phi\in X^{++}\}.\vspace{-1.5mm}\]
 \item The algebraic multiplicity of $\lambda_p$ is equal to one. Namely, $\lambda_p$ is simple.\vspace{-1.5mm}
 \item If there exists an eigenpair $(\lambda,\phi)$ of $\mathcal{L}$ with $\phi\in X^+\setminus\{(0,0)\}$, then $\lambda=\lambda_p$ and $\phi$ is a positive constant multiple of $\phi_p$.\vspace{-1.5mm}
 \item Suppose $a_{12}=a_{21}$, which implies that $\mathcal{L}$ is  self-adjoint. Then we have the variational characteristic $\lambda_p=\sup_{\|\phi\|_2=1}\langle\mathcal{L}[\phi],\phi\rangle$.
 \end{enumerate}
\end{proposition}

\begin{proof}We will prove conclusions (1)-(3) by two cases, Case 1: $d_1d_2>0$, and Case 2: $d_1=0$ or $d_2=0$.  Clearly, Case 2 is referred to as the partially degenerate case.

{\it Case 1: $d_1d_2>0$}. In this case, we note that conclusions (1)-(3) follow directly from \cite[Corollary 1.3 and Theorem 1.4]{SWZ}. In fact, it is easy to check that $H(x)$ is strongly irreducible in $[0,l]$ for any $l>0$. Hence it remains to show
  \bes\label{2.3}
  \frac1{\max_{[0,l]}\beta(x)-\beta(x)}\notin L^1([0,l]),
  \ees
where $\beta(x)$ is an eigenvalue of $H(x)$ and the maximum of real parts of all eigenvalues of $H(x)$.

Notice that $j'_i(x)=J_i(x)$ for $i=1,2$. Simple computations yield
 \bess \beta(x)&=&\dd\frac{-(d_1j_1(x)-a_{11}+d_2j_2(x)-a_{22})
 +\sqrt{(d_1j_1(x)-a_{11}-d_2j_2(x)+a_{22})^2+4a_{12}a_{21}}}2,\\[2mm]
 \beta'(x)&\le&0 ~ ~ {\rm and}~ ~ \beta'(0)<0,\eess
 which implies \eqref{2.3}, and conclusions (1)-(3) are derived in this case.

{\it Case 2: $d_1=0$ or $d_2=0$}. Without loss of generality, we suppose that $d_1=0<d_2$. By \cite[Lemma 2.6]{LLW}, the eigenvalue problem
\[d_2\int_0^lJ_2(x-y)\omega(y)\dy-d_2j_2(x)\omega+a_{22}\omega=\zeta\omega\]
has a principal eigenvalue $\zeta$ with a corresponding positive eigenfunction $\omega\in C([0,l])$. Let
\[\lambda^*_p=\frac{a_{11}+\zeta+\sqrt{(a_{11}-\zeta)^2+4a_{12}a_{21}}}2, ~ ~\phi_1=\frac{a_{12}\omega}{\lambda^*_p-a_{11}}, ~ ~\phi_2=\omega, ~ ~ \phi=(\phi_1,\phi_2)^T.\]
It is easy to see that $\lambda^*_p>a_{11}$ and  $\mathcal{L}[\phi]=\lambda^*_p\phi$.

Then we show $\lambda^*_p=\lambda_p$. From the definition of $\lambda_p$, we know $\lambda_p\le\lambda^*_p$. It thus remains to prove $\lambda_p\ge\lambda^*_p$. For any triplet $(\lambda,\psi_1,\psi_2)$ with $\psi=(\psi_1,\psi_2)\in X^{++}$ and $\mathcal{L}[\psi]\le\lambda\psi$. We shall prove $\lambda\ge\lambda^*_p$ which, combined with the definition of $\lambda_p$, leads to our desired result.

Denote $\int_0^lf(x)g(x)\dx$ by $\langle f,g\rangle$ for $f,g\in L^2([0,l])$. Then we have
\bess&&\langle\lambda^*_p\phi_1,\psi_1\rangle-a_{12}\langle\phi_2,\psi_1\rangle
=\langle\phi_1,a_{11}\psi_1\rangle\le\langle\phi_1,\lambda\psi_1-a_{12}\psi_2\rangle
=\langle\lambda\phi_1,\psi_1\rangle-a_{12}\langle\phi_1,\psi_2\rangle,
\eess
which leads to
\bes\label{2.4}(\lambda^*_p-\lambda)\langle\phi_1,\psi_1\rangle\le a_{12}\langle\phi_2,\psi_1\rangle-a_{12}\langle\phi_1,\psi_2\rangle.
\ees
Moreover,
\bess
\langle\lambda^*_p\phi_2-a_{21}\phi_1-a_{22}\phi_2,\psi_2\rangle&=&\left\langle\phi_2,\; d_2\int_0^lJ_2(x-y)\psi_2(y)\dy-d_2j_2(x)\psi_2\right\rangle\\[1mm]
&\le&\langle\phi_2,\lambda\psi_2-a_{21}\psi_1-a_{22}\psi_2\rangle,
\eess
which yields
 \bess
(\lambda^*_p-\lambda)\langle\phi_2,\psi_2\rangle\le a_{21}\langle\phi_1,\psi_2\rangle-a_{21}\langle\phi_2,\psi_1\rangle.
 \eess
Combining this with \eqref{2.4} gives
\[(\lambda^*_p-\lambda)\kk(\frac{\langle\phi_1,\psi_1\rangle}{a_{12}}
+\frac{\langle\phi_2,\psi_2\rangle}{a_{21}}\rr)\le0,\]
which, together with the fact that $a_{12}>0$, $a_{21}>0$, $\langle\phi_1,\psi_1\rangle>0$ and $\langle\phi_2,\psi_2\rangle>0$, arrives at $\lambda^*_p\le\lambda$. Therefore, $\lambda_p=\lambda^*_p$, and $\lambda_p$ is an eigenvalue of $\mathcal{L}$ with corresponding eigenfunction $\phi\in X^{++}$. That is, conclusion (1) is obtained. Then conclusions (2) and (3) can be deduced by \cite[Theorem 1.4]{SWZ}.

{\rm (4)} Assume $a_{12}=a_{21}$. For convenience, we denote $\lambda_0=\sup_{\|\phi\|_2=1}\langle\mathcal{L}[\phi],\phi\rangle$.
Clearly, $\lambda_0$ is well defined. It suffices to show that $\lambda_0$ is an eigenvalue of $\mathcal{L}$ with an eigenfunction in $X^+\setminus\{(0,0)\}$.

To this end, we first prove $\lambda_0>\beta(0)$.
Let
\[\alpha=\frac{2a_{12}}{{d_2}/2+a_{22}-{d_1}/2-a_{11}+\sqrt{({d_1}/2
+a_{11}-{d_2}/2-a_{22})^2+4a^2_{12}}}.\]
By \cite[Lemma 2.6]{LLW}, there is a positive function $\varphi_1\in C([0,l])$ with  $\int_0^l(1+\alpha^2)\varphi^2_1\dx=1$ such that
 \bess
&&\int_0^l\!\int_0^l\big[d_1J_1(x\!-\!y)
\!+\!\alpha^2d_2J_2(x\!-\!y)\big]
\varphi_1(y)\varphi_1(x)\dy\dx
\!-\!\int_0^l\!\big[d_1j_1(x)+\alpha^2d_2j_2(x)\big]\varphi^2_1\dx\\
&\ge&-\frac{d_1+\alpha^2d_2}2
\int_0^l\varphi^2_1\dx.
 \eess
Let $\varphi=(\varphi_1,\alpha\varphi_1)^T$ be the testing function. Clearly, $\|\varphi\|_2=1$. Simple computations yield
  \bess
\lambda_0&=&\sup_{\|\psi\|_2=1}\langle\mathcal{L}[\psi],\psi\rangle\ge \langle\mathcal{L}[\varphi],\varphi\rangle\\
&=&\int_0^l\int_0^l\big[d_1J_1(x-y)+\alpha^2d_2J_2(x-y)\big]\varphi_1(y)\varphi_1(x)\dy\dx
-\int_0^l\big[d_1j_1(x)+\alpha^2d_2j_2(x)\big]\varphi^2_1\dx\\
&&+(\alpha a_{12}-a_{11})\int_0^l\varphi^2_1\dx\dx+\alpha(a_{21}-\alpha a_{22})\int_0^l\varphi^2_1\dx\\
&>&\int_0^l\big[2a_{12}\alpha-{d_1}/2-a_{11}
-({d_2}/2+a_{22})\alpha^2\big]\varphi^2_1\dx=\beta(0).\eess
Thus $\lambda_0>\beta(0)$.

By virtue of $a_{12}=a_{21}$ and the definition of $\lambda_0$, we see that $\langle\lambda_0\varphi-\mathcal{L}[\varphi],\psi\rangle$ is bilinear, symmetric and $\langle\lambda_0\varphi-\mathcal{L}[\varphi],\varphi\rangle\ge0$. So by Cauchy-Schwarz inequality, we have
 \bess
|\langle\lambda_0\varphi-\mathcal{L}[\varphi],\psi\rangle|\le\langle\lambda_0\varphi-\mathcal{L}[\varphi],\varphi
\rangle^{\frac12}\langle\lambda_0\psi-\mathcal{L}[\psi],\psi\rangle^{\frac12}\le \langle\lambda_0\varphi-\mathcal{L}[\varphi],\varphi\rangle^{\frac12}\|\lambda_0I-\mathcal{L}\|^{\frac{1}{2}}\|\psi\|_2,
\eess
which yields $\|\lambda_0\varphi-\mathcal{L}[\varphi]\|_2\le\langle\lambda_0\varphi
 -\mathcal{L}[\varphi],\varphi\rangle^{\frac12}\|\lambda_0I-\mathcal{L}\|^{\frac{1}{2}}$.
Together with the definitions of $\lambda_0$ and $\mathcal{L}$, we derive that there exists a nonnegative sequence $\{\varphi^n\}$ with $\|\varphi^n\|_2=1$ such that
  \bes\label{2.6}
  \|\lambda_0\varphi^n-\mathcal{L}[\varphi^n]\|_2\to0 ~ ~{\rm  as ~}  n\to\yy.\ees

For convenience, let $\mathcal{T}[\varphi]=(\lambda_0I-H)[\varphi]$. By Arzel{\`a}-Ascoli Theorem, $\mathcal{P}$ is compact and maps $E$ to $X$. Thus there exists a subsequence of $\{\varphi^n\}$, still denoted by itself, such that $\mathcal{P}[\varphi^n]\to \bar\varphi$ for some $\bar\varphi\in X$. Moreover, owing to $\lambda_0>\beta(0)$, we have that $\mathcal{T}$ has a bounded and linear inverse $\mathcal{T}^{-1}$. Define $\mathcal{T}^{-1}[\bar\varphi]=\theta$. Clearly, $\theta\in X$. So $\lim_{n\to\yy}\mathcal{T}^{-1}[\mathcal{P}[\varphi^n]]
=\mathcal{T}^{-1}[\bar\varphi]=\theta$ in $X$.  Notice that
 \[\mathcal{T}^{-1}[\mathcal{P}[\varphi^n]]-\varphi^n
=\mathcal{T}^{-1}[\mathcal{P}[\varphi^n]-\mathcal{T}[\varphi^n]]
 =\mathcal{T}[\mathcal{L}[\varphi^n]-\lambda_0\varphi^n].\]
 Thanks to \eqref{2.6}, $\lim_{n\to\yy}\varphi^n=\theta$ in $E$, which combined with the fact that $\varphi^n$ is nonnegative and $\theta\in X$, leads to $\theta\in X^+$. Therefore, $\mathcal{T}^{-1}[\mathcal{P}[\theta]]=\theta$, namely, $\mathcal{L}[\theta]=\lambda_0\theta$. Noticing $\|\theta\|_2=1$, we see that $\lambda_0$ is an eigenvalue of $\mathcal{L}$ with an eigenfunction $\theta\in X^+\setminus\{(0,0)\} $. Then by conclusion (3), $\lambda_p=\lambda_0$. The proof is complete.
\end{proof}

Then we investigate the dependence of $\lambda_p$ on interval $[0,l]$ and diffusion coefficients $d_1$ and $d_2$, respectively. Let
\[A=\begin{pmatrix}
  a_{11} & a_{12} \\
   a_{21} & a_{22}
    \end{pmatrix}, ~ ~
    B=\begin{pmatrix}
  -d_1/2+a_{11} & a_{12} \\
  a_{21} & -d_2/2+a_{22}
  \end{pmatrix}.\]
Direct computations show there exist $\gamma_A, \gamma_B\in\mathbb{R}$, $\theta_A>0$ and $\theta_B>0$ satisfying
\bes\left\{\begin{aligned}\label{2.7}
&\gamma_A=\frac{a_{11}+a_{22}+\sqrt{(a_{11}+a_{22})^2
+4[a_{12}a_{21}-a_{11}a_{22}]}}2,\\[1mm]
&\gamma_B=\frac{a_{11}-\frac{d_1}2+a_{22}-\frac{d_2}2
+\sqrt{(a_{11}-\frac{d_1}2+a_{22}-\frac{d_2}2)^2+4\big[a_{12}a_{21}-(a_{11}-\frac{d_1}2)
(a_{22}-\frac{d_2}2)\big]}}2,\\[1mm] &\theta_A=\frac{a_{12}}{\gamma_A-a_{11}},\;\;
\theta_B=\frac{a_{12}}{\gamma_B+\df{d_1}{2}-a_{11}}, ~ (\gamma_AI-A)(\theta_A,1)^T=0, ~ ~ (\gamma_BI-B)(\theta_B,1)^T=0.
    \end{aligned}\right.\quad
    \ees

The following lemma will be often used in our later arguments.

\begin{lemma}\label{l2.1}Let $\lambda_p$ be the principal eigenvalue of \eqref{2.1}. Then the following statements are valid.\vspace{-1.5mm}
 \begin{enumerate}[$(1)$]
\item If there exist $\phi=(\phi_1,\phi_2)^T\in X$ with $\phi_1,\phi_2\ge,\not\equiv0$ and $\lambda\in\mathbb{R}$ such that $\mathcal{L}[\phi]\le\lambda\phi$, then $\lambda_p\le\lambda$. Moreover, $\lambda_p=\lambda$ only if $\mathcal{L}[\phi]=\lambda\phi$.\vspace{-1.5mm}
\item If there exist $\phi=(\phi_1,\phi_2)^T\in X^+\setminus\{(0,0)\}$ and $\lambda\in\mathbb{R}$ such that $\mathcal{L}[\phi]\ge\lambda\phi$, then $\lambda_p\ge\lambda$. Moreover, $\lambda_p=\lambda$ only if $\mathcal{L}[\phi]=\lambda\phi$.\vspace{-1.5mm}
   \end{enumerate}
   \end{lemma}

\begin{proof}    By arguing as in the proof of \cite[Lemma 2.2]{DN8} with some obvious modifications, we can prove this result. So the details are ignored. \end{proof}

It is worthy mentioning that in Lemma \ref{l2.1}(2), we only need $\phi=(\phi_1,\phi_2)\in X^+\setminus\{(0,0)\}$ which implies that one of $\phi_1$ and $\phi_2$ is allowed to be identical to zero. This will be used later.

Now we are in the position to show the dependence of $\lambda_p$ on interval $[0,l]$, and thus rewrite $\lambda_p$ as $\lambda_p(l)$ to stress the relationship of $\lambda_p$ about $[0,l]$. We note that unlike those arguments in the proofs of \cite[Proposition 3.4]{CDLL} and \cite[Proposition 2.7]{NV}, the methods we use here are elementary analysis without resorting to variational characteristic. So we don't assume $a_{12}=a_{21}$ in the following result.

\begin{proposition}\label{p2.2} Let $\lambda_p(l)$ be the principal eigenvalue of \eqref{2.1}. Then the following results hold.\vspace{-1.5mm}
\begin{enumerate}[$(1)$]
\item $\lambda_p(l)$ is continuous and strictly increasing with respect to $l>0$.\vspace{-1.5mm}
\item $\lim_{l\to\yy}\lambda_{p}(l)=\gamma_A$,  where $\gamma_A$ is given by \eqref{2.7}.\vspace{-1.5mm}
\item $\lim_{l\to0}\lambda_p(l)=\gamma_B$, where $\gamma_B$ is given by \eqref{2.7}.\vspace{-1.5mm}
\end{enumerate}
\end{proposition}

\begin{proof}
{\rm (1)} This conclusion can be obtained by adopting a similar approach as in \cite[Proposition 2.3]{DN8}, and thus the details are omitted here.

{\rm (2)} Recall that $\gamma_A$ and $\theta_A$ are given by \eqref{2.7}. Define $\bar{\varphi}=(\theta_A,1)^T$. We claim that $\mathcal{L}[\bar\varphi]\le \gamma_A\bar{\varphi}$ for all $l>0$ which, combined with Lemma \ref{l2.1}, yields
  \bes\label{2.8}
  \lambda_p(l)\le\gamma_A\;\;\;\text{for all}\;\,l>0.
 \ees

Now we prove $\mathcal{L}[\bar\varphi]\le \gamma_A\bar{\varphi}$ for all $l>0$. Simple calculations lead to
 \bess
&&d_1\int_0^lJ_1(x-y)\theta_A\dy-d_1j_1(x)\theta_A+a_{11}\theta_A+a_{12}\le a_{11}\theta_A+a_{12}=\gamma_A\theta_A,\\
&&d_2\int_0^lJ_2(x-y)\dy-d_2j_2(x)+a_{21}\theta_A+a_{22}\le a_{21}\theta_A+a_{22}=\gamma_A.
 \eess
Thus our claim holds and \eqref{2.8} is obtained.

Define $\underline{\varphi}=(\underline{\varphi}_1(x),\underline\varphi_2(x))^T$ with $\underline{\varphi}_1(x)=\theta_A\xi(x)$, $\underline{\varphi}_2(x)=\xi(x)$ and $\xi(x)=\min\kk\{1,\, 2(l-x)/l\rr\}$.
We shall show that for any small $\ep>0$ there exists $l_{\ep}>0$ such that when $l>4l_{\ep}$ there holds:
 \bes\label{2.9}
 \mathcal{L}[\underline{\varphi}]\ge (\gamma_A-\max\{d_1,d_2\}\ep)\underline{\varphi} ~ ~{\rm for ~ }x\in[0,l],
\ees
which, by Lemma \ref{l2.1}, arrives at $\lambda_p(l)\ge \gamma_A-\ep$ for $l\ge 4l_{\ep}$. Then by the arbitrariness of $\ep$, we have $
  \liminf_{l\to\yy}\lambda_p(l)\ge\gamma_A$.

Next we prove \eqref{2.9}. We first consider the case $x\in[0,l/4]$.
Direct calculations yield that
\bess
&&d_1\int_0^lJ_1(x-y)\underline{\varphi}_1(y)\dy-d_1j_1(x)\underline{\varphi}_1+a_{11}\underline{\varphi}_1+a_{12}\underline{\varphi}_2\\
&\ge& d_1\theta_A\int_0^{l/2}J_1(x-y)\dy-d_1j_1(x)\theta_A+a_{11}\theta_A+a_{12}\\
&=&-d_1\theta_A\int_{l/2}^{\yy}J_1(x-y)\dy+a_{11}\theta_A+a_{12}\\
&\ge& -d_1\theta_A\ep+a_{11}\theta_A+a_{12}=(\gamma_A-d_1\ep)\theta_A\ge(\gamma_A-d_1\ep)\underline{\varphi}_1,
\eess
provided that $l$ is large enough such that $\int_{l/4}^{\yy}J_1(y)\dy\le \ep$. Similarly,
\[d_2\int_0^lJ_2(x-y)\underline{\varphi}_2(y)\dy-d_2j_2(x)\underline{\varphi}_2
+a_{21}\underline{\varphi}_1+a_{22}\underline{\varphi}_2
\ge(\gamma_A-d_2\ep)\underline{\varphi}_2.\]

Then we consider the case $x\in[l/4,l]$. In view of \cite[Lemma 7.3]{DN3} with $l_2=l$ and $l_1=l/2$, for any small $\ep>0$ there exists a $l_{\ep}>0$ such that for all $l\ge4l_{\ep}$,
\bess
\int_0^lJ_i(x-y)\xi(y)\dy\ge(1-\ep)\xi(x) ~ ~ {\rm for } ~ i=1,2, ~ x\in[l/4,l].
\eess
Using this estimate, we have
 \bess
d_1\int_0^lJ_1(x-y)\underline{\varphi}_1(y)\dy-d_1j_1(x)\underline{\varphi}_1
+a_{11}\underline{\varphi}_1+a_{12}\underline{\varphi}_2
&\ge& d_1(1-\ep)\underline{\varphi}_1-d_1\underline{\varphi}_1+a_{11}\underline{\varphi}_1+a_{12}\underline{\varphi}_2\\
&=&(-d_1\ep+a_{11})\underline{\varphi}_1+a_{12}\underline{\varphi}_2\\
&=&(\gamma_A-d_1\ep)\underline{\varphi}_1.
\eess
Similarly,
\[d_2\int_0^lJ_2(x-y)\underline{\varphi}_2(y)\dy
-d_2j_2(x)\underline{\varphi}_2+a_{21}\underline{\varphi}_1
+a_{22}\underline{\varphi}_2\ge(\gamma_A-d_2\ep)\underline{\varphi}_2.\]

Hence \eqref{2.9} holds and $\liminf_{l\to\yy}\lambda_p(l)\ge\gamma_A$. Then due to \eqref{2.8}, the conclusion (2) is obtained.

{\rm (3)} Recall that $\gamma_B$ and $\theta_B$ are determined in \eqref{2.7}. Let $\underline\psi=(\theta_B,1)^T$. We claim that $\mathcal{L}[\psi]\ge\gamma_B\psi$ for all $l>0$. In fact, it is easy to verify that $\int_0^lJ_i(x-y)\dy-j_i(x)\ge -\frac12$. This, combined with \eqref{2.7}, allows us to derive
\bess
d_1\int_0^lJ_1(x-y)\dy\theta_B-d_1j_1(x)\theta_B-a_{11}\theta_B+a_{12}
\ge -\frac{d_1}2\theta_B+a_{11}\theta_B+a_{12}=\gamma_B\theta_B.
\eess
Similarly,
 \[d_2\int_0^lJ_2(x-y)\dy-d_2j_2(x)+a_{21}\theta_B+a_{22}\ge\gamma_B.\]
Therefore, our claim is valid. It then follows from Lemma \ref{l2.1} that $\lambda_p(l)\ge\gamma_B$ for all $l>0$.

Define $\rho(l)=\max\limits_{i=1,2}\kk\{\frac{d_i}2-d_i\int_{l}^{\yy}J_i(y)\dy\rr\}$.
Clearly, $\rho(l)\to0$ as $l\to0$.
It is not hard to show
  \bess
&&d_1\int_0^lJ_1(x-y)\dy\theta_B-d_1j_1(x)\theta_B+a_{11}\theta_B+a_{12}\\
&=&\frac{-d_1}2\theta_B+a_{11}\theta_B+a_{12}
+\kk(\frac{d_1}2-d_1\int_{l}^{\yy}J_1(y)\dy\rr)\theta_B\le(\gamma_B+\rho(l))\theta_B.
 \eess
Analogously,
 \[d_2\int_0^lJ_2(x-y)\dy-d_2j_2(x)+a_{21}\theta_B+a_{22}\le\gamma_B+\rho(l).\]
Using Lemma \ref{l2.1} again, we have $\lambda_p(l)\le \gamma_B+\rho(l)$, which implies $\limsup_{l\to0}\lambda_p(l)\le \gamma_B$. Together with $\lambda_p(l)\ge\gamma_B$ for all $l>0$, we finish the proof of conclusion (3). The proof is complete.
\end{proof}

Then we investigate the dependence of $\lambda_p$ on diffusion coefficients $d_1$ and $d_2$. So we rewrite $\lambda_p$ as a binary function $\lambda_p(d_1,d_2)$ which, by Proposition \ref{p2.1}, is well defined on $[0,\yy)\times[0,\yy)\setminus\{(0,0)\}$.

\begin{proposition}\label{p2.3} Let $\lambda_p(d_1,d_2)$ be given as above. Then the following statements are valid.\vspace{-1.5mm}
\begin{enumerate}[$(1)$]
  \item $\lambda_p(d_1,d_2)$ is continuous with respect to $(d_1,d_2)\in [0,\yy)\times[0,\yy)\setminus\{(0,0)\}$.\vspace{-1.5mm}
  \item $\lambda_p(d_1,d_2)\to\gamma_A$ as $(d_1,d_2)\to (0,0)$, where $\gamma_A$ is given by \eqref{2.7}.\vspace{-1.5mm}
  \item If $a_{12}=a_{21}$, then $\lambda_p(d_1,d_2)$ is strictly decreasing in each variable $d_1>0$ and $d_2>0$.\vspace{-1.5mm}
  \item Fix $d_i>0$. Then $\lambda_p(d_1,d_2)\to\zeta_j$ as $d_j\to\yy$ where $i,j=1,2$, $i\neq j$ and $\zeta_j$ is the principal eigenvalue of
  \bess
  d_i\int_0^lJ_i(x-y)\omega(y)\dy-d_ij_i(x)\omega+a_{ii}\omega=\zeta\omega, ~ ~ x\in[0,l].\vspace{-1.5mm}\eess
  \item $\lambda_p(d_1,d_2)\to-\yy$ as $(d_1,d_2)\to(\yy,\yy)$.\vspace{-1.5mm}
\end{enumerate}
\end{proposition}

\begin{proof}
{\rm (1)} For any given $(\bar{d}_1,\bar{d}_2)$ and $(d_1,d_2)\in [0,\yy)\times[0,\yy)\setminus\{(0,0)\}$. Denote by $(\phi_1,\phi_2)^T$ the positive eigenfunction of $\lambda_p(d_1,d_2)$, and set $K=\max\limits_{i=1,2}\frac{\max_{[0,l]}\phi_i}{\min_{[0,l]}\phi_i}$. Direct computations yield
\bess
&&\bar{d}_1\int_0^lJ_1(x-y)\phi_1(y)\dy-\bar{d}_1j_1(x)\phi_1
+a_{11}\phi_1+a_{12}\phi_2\\
&=&\lambda_p(d_1,d_2)\phi_1+(\bar{d}_1-d_1)\int_0^lJ_1(x-y)\phi_1(y)\dy-(\bar d_1-d_1)j_1(x)\phi_1\\
&\le& \lambda_p(d_1,d_2)\phi_1+2|\bar{d}_1
-d_1|K\phi_1.
\eess
Similarly,
\bess
\bar{d}_2\int_0^lJ_2(x-y)\phi_2(y)\dy-\bar{d}_2j_2(x)\phi_2
+a_{21}\phi_1+a_{22}\phi_2
\le\lambda_p(d_1,d_2)\phi_2+2|\bar{d}_2
-d_2|K\phi_2.\eess
Thus, it follows from Lemma \ref{l2.1} that
\bes\label{2.12}
\lambda_p(\bar{d}_1,\bar{d}_2)\le \lambda_p(d_1,d_2)+2K(|\bar{d}_1-d_1|+|\bar{d}_2-d_2|),
\ees
Similar to the above, we have
 \[\lambda_p(\bar{d}_1,\bar{d}_2)\ge \lambda_p(d_1,d_2)-2K(|\bar{d}_1-d_1|+|\bar{d}_2-d_2|),\]
which, together with \eqref{2.12}, derives
\[|\lambda_p(\bar{d}_1,\bar{d}_2)-\lambda_p(d_1,d_2)|\le2K(|\bar{d}_1-d_1|+|\bar{d}_2-d_2|).\]
The continuity follows.

{\rm (2)} Let $\bar{\varphi}=(\theta_A,1)^T$ as in the proof of Proposition \ref{p2.2}. Direct computations show
 \bess
 d_1\int_0^lJ_1(x-y)\theta_A\dy-d_1j_1(x)\theta_A+a_{11}\theta_A+a_{12}&\ge& -d_1\theta_A+a_{11}\theta_A+a_{12}=(\gamma_A-d_1)\theta_A,\\
 d_2\int_0^lJ_2(x-y)\dy-d_2j_2(x)+a_{21}\theta_A+a_{22}&\ge & -d_2+a_{21}\theta_A+a_{22}=\gamma_A-d_2.
 \eess
Recalling Lemma \ref{l2.1}, we have $\lambda_p(d_1,d_2)\ge\gamma_A-(d_1+d_2)$, so  $\liminf_{(d_1,d_2)\to(0,0)}\lambda_p(d_1,d_2)\ge \gamma_A$. Moreover, owing to \eqref{2.8}, $\limsup_{(d_1,d_2)\to(0,0)}\lambda_p(d_1,d_2)\le \gamma_A$. Conclusion (2) is proved.

{\rm (3)} Note that $a_{12}=a_{21}$ in this statement. So by Proposition \ref{p2.1}, the variational characteristic holds. We only show the monotonicity of $\lambda_p(d_1,d_2)$ about $d_1$ since the other case is similar. We fix $d_2$ and choose any $0<\bar{d}_1<d_1$. Denote by $\phi=(\phi_1,\phi_2)^T$ the corresponding positive eigenfunction of $\lambda_p(d_1,d_2)$ with $\|\phi\|_2=1$. Firstly, using \cite[Lemma 2.6]{LLW}, we have
 \bess
\int_0^l\int_0^lJ_1(x-y)\phi_1(y)\phi_1(x)\dy\dx-\int_0^lj_1\phi^2_1\dx<0 ~ ~ {\rm for ~ all ~ }l>0.
 \eess
It then follows that
\bess
\lambda_p(d_1,d_2)&=&d_1\kk(\int_0^l\!\int_0^l\!J_1(x-y)\phi_1(y)\phi_1(x)\dy\dx
-\int_0^l\!j_1\phi^2_1\dx\rr)+\int_0^l\!(a_{11}\phi^2_1+a_{12}\phi_1\phi_2)\dx\\[1mm]
&&+d_2\!\int_0^l\!\int_0^l\!J_2(x-y)\phi_2(y)\phi_2(x)\dy\dx
-d_2\!\int_0^l\!j_2\phi^2_2\dx+\int_0^l\!(a_{21}\phi_1\phi_2+a_{22}\phi^2_2)\dx\\[1mm]
&<&\bar d_1\kk(\int_0^l\!\int_0^l\!J_1(x-y)\phi_1(y)\phi_1(x)\dy\dx
-\int_0^l\!j_1\phi^2_1\dx\rr)
+\int_0^l\!(a_{11}\phi^2_1+a_{12}\phi_1\phi_2)\dx\\[1mm]
&&+d_2\!\int_0^l\!\int_0^l\!J_2(x-y)\phi_2(y)\phi_2(x)\dy\dx
-d_2\!\int_0^l\!j_2\phi^2_2\dx+\int_0^l\!(a_{21}\phi_1\phi_2+a_{22}\phi^2_2)\dx\\[1mm]
&\le& \lambda_p(\bar{d}_1,d_2).
\eess
 The monotonicity is obtained.

(4) We only prove $\lambda_p(d_1,d_2)\to\zeta_1$ as $d_1\to\yy$ for the fix $d_2>0$, since the other case is parallel. Our arguments are inspired by \cite{Zhanglei}. Firstly, it follows from  \eqref{2.8} that  $\lambda_p(d_1,d_2)\le\gamma_A$. Let $\omega$ be the corresponding positive eigenfunction of $\zeta_1$ and $\underline{\varphi}=(0,\omega)^T$. It is easy to see that $\mathcal{L}[\underline{\varphi}]\ge\zeta_1\underline{\varphi}$, which implies $\lambda_p(d_1,d_2)\ge\zeta_1$. Consequently,
  \bes
  \zeta_1\le\lambda_p(d_1,d_2)\le\gamma_A\;\;\text{for all}\; d_1, d_2>0.
  \lbl{2.15a}\ees

In order to show $\lambda_p(d_1,d_2)\to\zeta_1$ as $d_1\to\yy$, it is sufficient to prove that for any sequence $\{d^n_1\}$ with $d^n_1\to\yy$ as $n\to\yy$, there is a subsequence, still denoted by itself, such that  $\lambda_p(d^n_1,d_2)\to\zeta_1$ as $n\to\yy$. For convenience, denote $\lambda_p(d^n_1,d_2)$ by $\lambda^n_p$ since we fix $d_2>0$. Let $\phi^n=(\phi^n_1,\phi^n_2)^T$ be the positive eigenfunction of $\lambda^n_p$ with $\|\phi^n\|_{X}=1$. Using this fact and \qq{2.15a} we deduce  that there exists a subsequence of $\{n\}$, still denoted by itself, such that $(\phi^n_1,\phi^n_2)$ converges weakly to $(\psi_1,\psi_2)$ with $\psi_i\in L^2([0,l])$, and $\lambda^n_p\to\lambda_{\yy}\ge\zeta_1$ as $n\to\yy$. Due to $\phi^n\in X^{++}$, we have $\psi_i\ge0$ for $i=1,2$.

Now we show that $\psi_1\equiv0$. Obviously,
\[d^n_1\int_0^lJ_1(x-y)\phi^n_1(y)\dy-d^n_1j_1(x)\phi^n_1+a_{11}\phi^n_1+a_{12}\phi^n_2=\lambda^n_p\phi^n_1, ~ ~ {\rm for ~ }x\in[0,l].\]
Dividing the above equation by $d^n_1$ and letting $n\to\yy$ one has
 \bes
 \int_0^lJ_1(x-y)\phi^n_1(y)\dy-j_1(x)\phi^n_1\to0 ~ ~ {\rm uniformly ~ in ~ }[0,l].\lbl{2.13}\ees
Since $\phi^n_1$ converges weakly to $\psi_1$ and operator $\int_0^{l}J_1(x-y)\phi^n_1(y)\dy: L^2([0,l])\to C([0,l])$ is compact, it follows that, as $n\to\yy$,
  \[\int_0^{l}J_1(x-y)\phi^n_1(y)\dy\to\int_0^{l}J_1(x-y)\psi_1(y)\dy ~ ~ {\rm uniformly ~ in ~ }[0,l].\]
This, combined with \eqref{2.13}, yields that, as $n\to\yy$,
  \[\phi^n_1\to \frac1{j_1(x)}\int_0^{l}J_1(x-y)\psi_1(y)\dy
  ~ ~ {\rm uniformly ~ in ~ }[0,l].\]
 By the uniqueness of weak limit,
 \[\psi_1(x)=\frac1{j_1(x)}\int_0^{l}J_1(x-y)\psi_1(y)\dy.\]
 If there exists some $x_0\in[0,l]$ such that $\psi_1(x_0)>0$, then it is not hard to show that $\psi_1(x)>0$ in $[0,l]$, which implies that $(0,\psi_1)$ is the principal eigenpair of the eigenvalue problem
 \bes
 \int_0^{l}J_1(x-y)\omega(y)\dy-j_1(x)\omega(x)=\xi\omega.
 \lbl{z.1}\ees
However, on the basis of \cite[Lemma 2.6]{LLW}, the principal eigenvalue $\xi$ of \eqref{z.1}  must be less than $0$. This contradiction implies $\psi_1\equiv0$. Thus $\phi^n_1\to0$ in $C([0,l])$ as $n\to\yy$.

Noticing that $\|\phi^n\|_{X}=1$, we have $\|\phi^n_2\|\to1$ as $n\to\yy$. Since $\phi^n_2\to\psi_2$ weakly in $L^2([0,l])$ and $\int_0^lJ_2(x-y)\phi^n_2(y)\dy: L^2([0,l])\to C([0,l])$ is compact, one has
\bes\label{2.15}\int_0^lJ_2(x-y)\phi^n_2(y)\dy\to\int_0^lJ_2(x-y)\psi_2(y)\dy ~ ~ {\rm uniformly ~ in ~ }[0,l].
\ees
Moreover, due to $\phi^n_1\to0$ in $C([0,l])$ as $n\to\yy$, one also has
\[d_2\int_0^lJ_2(x-y)\phi^n_2(y)\dy-d_2j_2(x)\phi^n_2
+a_{22}\phi^n_2-\lambda^n_p\phi^n_2=-a_{21}\phi^n_1\to0 ~ ~ {\rm uniformly ~ in}~ [0,l].\]
Since
 \bess
d_2j_2(x)-a_{22}+\lambda^n_p\ge\frac{d_2}2-a_{22}+\zeta_1>0,
 \eess
we have that, as $n\to\yy$,
 \bess\phi^n_2(x)-\frac{\dd d_2\int_0^lJ_2(x-y)\phi^n_2(y)\dy}{d_2j_2(x)-a_{22}+\lambda^n_p}\to0 ~ ~ {\rm uniformly ~ in}~ [0,l].\eess
This, combines with \eqref{2.15}, yields that, as $n\to\yy$,
\[\phi^n_2(x)\to\frac{\dd d_2\int_0^lJ_2(x-y)\psi_2(y)\dy}{d_2j_2(x)-a_{22}+\lambda_{\yy}} ~ ~ {\rm uniformly~ in ~ }[0,l].\]
Note that $\phi^n_2$ converges weakly to $\psi_2$. By the uniqueness of limit, we obtain
\bes\label{2.16}d_2\int_0^lJ_2(x-y)\psi_2(y)\dy-d_2j_2(x)\psi_2
+a_{22}\psi_2=\lambda_{\yy}\psi_2,\ees
and $\phi^n_2\to\psi_2$ in $C([0,l])$. Recall that $\|\phi^n_2\|_{C([0,l])}\to1$ as $n\to\yy$. So $\|\psi_2\|_{C([0,l])}=1$. Together with \eqref{2.16}, we easily derive that $\psi_2>0$ in $[0,l]$, which implies $\lambda_{\yy}=\zeta_1$. Thus conclusion (4) is proved.

(5) It can be seen from \cite[Lemma 2.6]{LLW} that, for $i=1,2$, the following eigenvalue problem
 \bess\int_0^lJ_i(x-y)\omega(y)\dy-j_i(x)\omega(x)=\lambda\omega(x), ~ ~ x\in[0,l]\eess
has a principal eigenpair $(\lambda_i,\omega_i)$ with $\omega_i$ positive and satisfying $\|\omega_i\|_{C([0,\,l])}=1$. Moreover, $\lambda_i\in(-1/2,0)$. Define
\bess
d=\min\{d_1,d_2\},~ \lambda=\min\{\lambda_1,\lambda_2\}, ~ \omega=(\omega_1,\omega_2)^T, ~ k=|a_{11}|+|a_{22}|+\dd\frac{a_{12}}{\min_{[0,\,l]}\omega_1}
+\frac{a_{21}}{\min_{[0,\,l]}\omega_2}.
 \eess
Simple computations yield
\bess
&&d_1\int_0^lJ_1(x-y)\omega_1(y)\dy-d_1j_1(x)\omega_1+a_{11}\omega_1+a_{12}\omega_2\\
&\le&\kk(d_1\lambda_1+|a_{11}|+\frac{a_{12}}{\min_{[0,\,l]}\omega_1}\rr)\omega_1\le(d_1\lambda_1+k)\omega_1.
\eess
Similarly,
\[d_2\int_0^lJ_2(x-y)\omega_2(y)\dy-d_2j_2(x)\omega_2+a_{21}\omega_1+a_{22}\omega_2\le(d_2\lambda_2+k)\omega_2.\]
Therefore, $\mathcal{L}[\omega]\le( d\lambda+k)\omega$. By Lemma \ref{l2.1}, $\lambda_p(d_1,d_2)\le d\lambda+k$. From the fact that $\lambda<0$ and $k$ is independent of $(d_1,d_2)$, we obtain conclusion (5).
The proof of Proposition \ref{p2.3} is complete.
\end{proof}

\section{Positive equilibrium solutions associated to \eqref{1.10}}\lbl{s3}

With the help of the results obtained in Section \ref{s2}, in this section, we discuss the positive equilibrium solutions associated to \eqref{1.10} which reads as
 \bes\left\{\begin{aligned}\label{2.17}
&d_1\int_0^lJ_1(x-y)\boldsymbol{u}(y)\dy-d_1j_1(x)\boldsymbol{u}-a\boldsymbol{u}+H(\boldsymbol{v})=0, & & x\in[0,l],\\
&d_2\int_0^lJ_2(x-y)\boldsymbol{v}(y)\dy-d_2j_2(x)\boldsymbol{v}-b\boldsymbol{v}+G(\boldsymbol{u})=0, & & x\in[0,l],
 \end{aligned}\right.
 \ees
where all parameters are positive, and condition {\bf(H1)} holds. In the remainder of this paper, let $\lambda_1(l)$ and $\lambda_2(l)$ be the principal eigenvalue of the following two eigenvalue problems, respectively,
  \bes\label{2.18}\left\{\begin{aligned}
&d_1\int_0^lJ_1(x-y)\phi_1(y)\dy-d_1j_1(x)\phi_1-a\phi_1
+H'(0)\phi_2=\lambda\phi_1, & &x\in[0,l],\\[1mm]
&d_2\int_0^lJ_2(x-y)\phi_2(y)\dy-d_2j_2(x)\phi_2
+G'(0)\phi_1-b\phi_2=\lambda\phi_2, & &x\in[0,l].
 \end{aligned}\right.
 \ees
  \bes\left\{\begin{aligned}
&\frac1{H'(0)}\left(d_1\int_0^lJ_1(x-y)\phi_1(y)\dy-d_1j_1(x)\phi_1
-a\phi_1\right)+\phi_2=\lambda\phi_1, & &x\in[0,l],\\[1mm]
&\frac1{G'(0)}\left(d_2\int_0^lJ_2(x-y)\phi_2(y)\dy-d_2j_2(x)\phi_2
-b\phi_2\right)+\phi_1=\lambda\phi_2,& &x\in[0,l].
 \end{aligned}\right.\lbl{2.19}
 \ees
It is easy to see that these two eigenvalue problems are not equivalent, and
 \bes\label{2.20}\left\{\begin{aligned}
& \frac{\lambda_1(l)}{\max\{H'(0),G'(0)\}}\le\lambda_2(l)\le
 \frac{\lambda_1(l)}{\min\{H'(0),G'(0)\}}, & &{\rm if ~ }\lambda
 _1\ge0,\\[1mm]
& \frac{\lambda_1(l)}{\min\{H'(0),G'(0)\}}\le\lambda_2(l)\le
 \frac{\lambda_1(l)}{\max\{H'(0),G'(0)\}},& &{\rm if ~ }\lambda
 _1<0,
 \end{aligned}\right.
 \ees
which clearly implies that $\lambda_1(l)$ and $\lambda_2(l)$ have the same sign.  For clarity, in this paper, we usually use $\lambda_1(l)$ to study dynamics of \eqref{1.10}, and only utilize $\lambda_2(l)$ when discussing the effect of diffusion coefficients $d_1$ and $d_2$ since, by Proposition \ref{p2.3}, the monotonicity of $\lambda_2(l)$ holds.

Below is a maximum principle for \eqref{2.17} that will be used in the coming analysis.

 \begin{lemma}\label{l2.3}Let $(\boldsymbol{u}_i,\boldsymbol{v}_i)\in X^{++}$ for $i=1,2$ and satisfy
 \bes
 \label{2.21}\left\{\begin{aligned}
&d_1\int_0^lJ_1(x-y)\boldsymbol{u}_1(y)\dy-d_1j_1(x)\boldsymbol{u}_1-a\boldsymbol{u}_1+H(\boldsymbol{v}_1)\le0, & & x\in[0,l]\\
&d_2\int_0^lJ_2(x-y)\boldsymbol{v}_1(y)\dy-d_2j_2(x)\boldsymbol{v}_1-b\boldsymbol{v}_1+G(\boldsymbol{u}_1)\le0, & & x\in[0,l],
 \end{aligned}\right.
 \ees
 and
  \bes\label{2.22}\left\{\begin{aligned}
&d_1\int_0^lJ_1(x-y)\boldsymbol{u}_2(y)\dy-d_1j_1(x)\boldsymbol{u}_2-a\boldsymbol{u}_2+H(\boldsymbol{v}_2)\ge0, & & x\in[0,l]\\
&d_2\int_0^lJ_2(x-y)\boldsymbol{v}_2(y)\dy-d_2j_2(x)\boldsymbol{v}_2-b\boldsymbol{v}_2+G(\boldsymbol{u}_2)\ge0, & & x\in[0,l],
 \end{aligned}\right.
 \ees
respectively. Then $(\boldsymbol{u}_1-\boldsymbol{u}_2,\boldsymbol{v}_1-\boldsymbol{v}_2)\in X^{+}$. Moreover, if one of the above four inequalities is strict at some point $x_0\in[0,l]$, then $(\boldsymbol{u}_1-\boldsymbol{u}_2,\boldsymbol{v}_1-\boldsymbol{v}_2)\in X^{++}$.
 \end{lemma}

\begin{proof} {\it Step 1}: {\it The proof of $(\boldsymbol{u}_1-\boldsymbol{u}_2,\boldsymbol{v}_1-\boldsymbol{v}_2)\in X^{+}$}.  Since $(\boldsymbol{u}_i,\boldsymbol{v}_i)\in X^{++}$ for $i=1,2$, then
 \[\ud\kappa =\inf\{\kappa>1: (\kappa \boldsymbol{u}_1-\boldsymbol{u}_2,\kappa \boldsymbol{v}_1-\boldsymbol{v}_2)\in X^+\}\]
is well defined and $\ud\kappa \ge1$. Clearly, $(\ud\kappa  \boldsymbol{u}_1-\boldsymbol{u}_2,\ud\kappa  \boldsymbol{v}_1-\boldsymbol{v}_2)\in X^+$. If $\ud\kappa >1$, then there exists a point $x_1\in[0,l]$ such that $\ud\kappa \boldsymbol{u}_1(x_1)=\boldsymbol{u}_2(x_1)$ or $\ud\kappa \boldsymbol{v}_1(x_1)=\boldsymbol{v}_2(x_1)$. We first prove that $\ud\kappa \boldsymbol{u}_1(x_1)=\boldsymbol{u}_2(x_1)$ is impossible. Assume on the contrary that $\ud\kappa \boldsymbol{u}_1(x_1)=\boldsymbol{u}_2(x_1)$.

{\it Case 1}: $\ud\kappa \boldsymbol{v}_1(x_1)>\boldsymbol{v}_2(x_1)$. In view of the first inequalities of \eqref{2.21} and \eqref{2.22}, we have
 \bess
 &&\ud\kappa d_1\int_0^lJ_1(x_1-y)\boldsymbol{u}_1(y)\dy-d_1j_1(x_1)\ud\kappa \boldsymbol{u}_1(x_1)
 -a\ud\kappa \boldsymbol{u}_1(x_1)+\ud\kappa H(\boldsymbol{v}_1(x_1))\le0,\\
 &&d_1\int_0^lJ_1(x_1-y)\boldsymbol{u}_2(y)\dy-d_1j_1(x_1)\ud\kappa \boldsymbol{u}_1(x_1)
 -a\ud\kappa \boldsymbol{u}_1(x_1)+H(\boldsymbol{v}_2(x_1))\ge0,
 \eess
which, together with $\ud\kappa  \boldsymbol{u}_1(x)\ge \boldsymbol{u}_2(x)$ in $[0,l]$, implies $H(\boldsymbol{v}_2(x_1))\ge\ud\kappa H(\boldsymbol{v}_1(x_1))$.
However, thanks to the assumption on $H$, $\ud\kappa >1$ and $\ud\kappa \boldsymbol{v}_1(x_1)>\boldsymbol{v}_2(x_1)$, we easily obtain $H(\boldsymbol{v}_2(x_1))<\ud\kappa H(\boldsymbol{v}_1(x_1))$. This is a contradiction.

{\it Case 2}: $\ud\kappa \boldsymbol{v}_1(x_1)=\boldsymbol{v}_2(x_1)$. Similar to the above, it can be derived that $G(\boldsymbol{u}_2(x_1))\ge\ud\kappa G(\boldsymbol{u}_1(x_1))$. Due to the assumption on $G$, $\ud\kappa>1$ and $\ud\kappa \boldsymbol{u}_1(x_1)=\boldsymbol{u}_2(x_1)$, we also can derive a contradiction.

Similarly, $\ud\kappa \boldsymbol{v}_1(x_1)=\boldsymbol{v}_2(x_1)$ is impossible. Hence,
$\ud\kappa =1$ and thus $(\boldsymbol{u}_1-\boldsymbol{u}_2,\boldsymbol{v}_1-\boldsymbol{v}_2)\in X^{+}$.

{\it Step 2}: {\it Proof of $(\boldsymbol{u}_1-\boldsymbol{u}_2,\boldsymbol{v}_1-\boldsymbol{v}_2)\in X^{++}$}. We only handle the case where the first inequality in \eqref{2.21} is strict at $x_0\in [0,l]$ since other cases can be done by the similar way. Argue indirectly that $(\boldsymbol{u}_1-\boldsymbol{u}_2,\boldsymbol{v}_1-\boldsymbol{v}_2)\notin X^{++}$. Then there is a point $x_2\in[0,l]$ such that $\boldsymbol{u}_1(x_2)=\boldsymbol{u}_2(x_2)$ or $\boldsymbol{v}_1(x_2)=\boldsymbol{v}_2(x_2)$. Define
 \[\Sigma=\{x\in[0,l]:\,\boldsymbol{u}_1(x)=\boldsymbol{u}_2(x)\},\;\;\;\Pi=\{x\in[0,l]:\,\boldsymbol{v}_1(x)=\boldsymbol{v}_2(x)\}.\]
Then at least one of $\Sigma$ and $\Pi$ is nonempty. We first consider the case that $\Sigma\not=\emptyset$.

If $x_0\in\Sigma$, i.e., $\boldsymbol{u}_1(x_0)=\boldsymbol{u}_2(x_0)$. As above, it can be deduced by the first inequalities of \eqref{2.21} and \eqref{2.22} that $H(\boldsymbol{v}_2(x_0))>H(\boldsymbol{v}_1(x_0))$, which clearly contradicts the monotonicity of $H$ and the fact $\boldsymbol{v}_2(x_0)\le \boldsymbol{v}_1(x_0)$.

If $x_0\not\in\Sigma$, then $\boldsymbol{u}_1(x_0)>\boldsymbol{u}_2(x_0)$. Choose $x_2\in\Sigma$, i.e.,  $\boldsymbol{u}_1(x_2)=\boldsymbol{u}_2(x_2)$. Clearly, $x_2\neq x_0$. We assume that $x_2>x_0$ without loss of generality. Then there exists a point $x_3\in(x_0,x_2]$ such that $\boldsymbol{u}_1(x_3)=\boldsymbol{u}_2(x_3)$ and $\boldsymbol{u}_1>\boldsymbol{u}_2$ in $[x_0,x_3)$. Thus, making use of the condition {\bf (J)} and the fact that $\boldsymbol{u}_1\ge \boldsymbol{u}_2$ in $[0,l]$, we have $\int_0^lJ_1(x_3-y)\boldsymbol{u}_2(y)\dy<\int_0^lJ_1(x_3-y)\boldsymbol{u}_1(y)\dy$. However, analogously, it can be derived by the first inequalities of \eqref{2.21} and \eqref{2.22} that
$\int_0^lJ_1(x_3-y)\boldsymbol{u}_2(y)\dy\ge\int_0^lJ_1(x_3-y)\boldsymbol{u}_1(y)\dy$. This is a contradiction.

Now we consider the case $\Sigma=\emptyset$, i.e., $\boldsymbol{u}_1>\boldsymbol{u}_2$ in $[0,l]$. Then $\Pi\not=\emptyset$. Choose $x_4\in\Pi$, i.e., $\boldsymbol{v}_1(x_4)=\boldsymbol{v}_2(x_4)$. Notice that $G'(z)>0$ and $\boldsymbol{u}_1(x_4)>\boldsymbol{u}_2(x_4)$. It then follows from the second equalities of \eqref{2.21} and \eqref{2.22} that
 $\int_0^lJ_2(x_4-y)\boldsymbol{v}_1(y)\dy<\int_0^lJ_2(x_4-y)\boldsymbol{v}_2(y)\dy$,
 which clearly contradicts $\boldsymbol{v}_1\ge \boldsymbol{v}_2$ in $[0,l]$. The proof is ended.
 \end{proof}

We now give the result concerning the bounded positive solution of \eqref{2.17}. Note that our arguments are different from those in proofs of \cite[Lemmas 3.10 and 3.11]{WD1}, \cite[Proposition 3.4]{DN8} and \cite[Proposition 2.10]{NV}. Especially, the lack of shifting invariance property of \eqref{2.17} brings some difficulties in the proof of the following assertion $(\boldsymbol{u}_l,\boldsymbol{v}_l)\to(U,V)$ as $l\to\yy$.

\begin{lemma}\label{l2.4} Let $\lambda_1(l)$ be defined as above. Then the following statements are valid.\vspace{-1.5mm}
 \begin{enumerate}[$(1)$]
 \item If $\lambda_1(l)>0$, then problem \eqref{2.17} has a unique bounded positive solution $(\boldsymbol{u},\boldsymbol{v})\in X^{++}$ and $(U-\boldsymbol{u},V-\boldsymbol{v})\in X^{++}$. Denote $(\boldsymbol{u},\boldsymbol{v})$ by $(\boldsymbol{u}_l,\boldsymbol{v}_l)$. Then $(\boldsymbol{u}_l,\boldsymbol{v}_l)$ is strictly increasing for large $l>0$ and $(\boldsymbol{u}_l,\boldsymbol{v}_l)\to(U,V)$ locally uniformly in $[0,\yy)$ as $l\to\yy$.\vspace{-1.5mm}
 \item If $\lambda_1(l)\le0$, then $(0,0)$ is the unique nonnegative solution of \eqref{2.17}.\vspace{-1.5mm}
 \end{enumerate}
 \end{lemma}

{\setlength\arraycolsep{2pt}

\begin{proof}(1) In view of Proposition \ref{p2.2}, we have that if $\lambda_1(l)>0$, then $\gamma_A>0$, where $\gamma_A$ is defined in \eqref{2.7} and the matrix $A$ here is composed of $a_{11}=-a$, $a_{12}=H'(0)$, $a_{21}=G'(0)$ and $a_{22}=-b$. It is easy to see that $\gamma_A>0$ if and only if $\mathcal{R}_0>1$.

{\it Step 1}: {\it The existence}. Define an operator ${\it\Gamma}$: $X^+\to X^+$ by \[{\it\Gamma}[\phi]=\begin{pmatrix}
 \dd\frac1{d_1j_1(x)+a}\left(d_1\int_0^lJ_1(x-y)\phi_1(y)\dy+H(\phi_2)\right) \\[4mm] \dd\frac1{d_2j_2(x)+b}\left(d_2\int_0^lJ_2(x-y)\phi_2(y)\dy+G(\phi_1)\right)
  \end{pmatrix}.\]
Clearly, ${\it\Gamma}$ is increasing in $X^+$. Simple computations show
 \bes
 \frac1{d_1j_1(x)+a}\left(d_1\int_0^lJ_1(x-y)U\dy+H(V)\right)&\le&
 \frac1{d_1j_1(x)+a}\left[d_1j_1(x)U+H(V)\right]\nonumber\\
 &=&\frac1{d_1j_1(x)+a}\left[d_1j_1(x)U+aU\right]\nonumber\\[1mm]
 &=&U,\;\;\;x\in[0,l],\lbl{2.23}\\
 \frac1{d_2j_2(x)+b}\left(d_2\int_0^lJ_2(x-y)V\dy+G(U)\right)&\le&
 V,\;\;\;x\in[0,l],\lbl{2.24}
 \ees
which implies that ${\it\Gamma}[(U,V)]\le(U,V)$. Moreover, \eqref{2.23} and \eqref{2.24} are strict when $x<l$ and near $l$.

Let $\phi=(\phi_1,\phi_2)\in X^{++}$ be the corresponding eigenfunction of $\lambda_1(l)$ with $\|\phi\|_{X}=1$. We claim that if $\ep$ is sufficiently small, then ${\it\Gamma}[\ep\phi]\ge\ep\phi$. In fact, the direct calculation yields
 \bess
 &&\frac1{d_1j_1(x)+a}\left[d_1\int_0^lJ_1(x-y)\phi_1(y)\dy+H(\phi_2)\right]
 -\ep\phi_1\\[1mm]
 &\ge&\frac\ep{d_1j_1(x)+a}\left[\lambda_1(l)\phi_1+d_1j_1(x)\phi_1+a\phi_1
 +H(\ep)/\ep-H'(0)\right]-\ep\phi_1\\[1mm]
 &\ge&\frac\ep{d_1j_1(x)+a}\left[\lambda_1(l)\phi_1+H(\ep)/\ep-H'(0)\right]\ge 0
 \eess
provided that $\ep$ is small enough. Similarly,
 \[\frac1{d_2j_2(x)+b}\left[d_2\int_0^lJ_2(x-y)\ep\phi_2(y)\dy
 +G(\ep\phi_1)\right]\ge\ep\phi_2\]
 with $\ep$ small enough. Thus our claim holds.

Then by an iteration or upper-lower solution method, problem \eqref{2.17} has at least one  solution $(\boldsymbol{u},\boldsymbol{v})$ satisfying $(\ep\phi_1,\ep\phi_2)\le(\boldsymbol{u},\boldsymbol{v})\le (U,V)$ in $[0,l]$.

{\it Step 2}: {\it The continuity}. It will be proved that $(\boldsymbol{u},\boldsymbol{v})$ is continuous in $[0,l]$ by using the implicit function theorem and some basic analysis. Define
 \bess
 &Q_1(x)=\dd d_1\int_0^lJ_1(x-y)\boldsymbol{u}(y)\dy, ~ ~ Q_2(x)=d_2\int_0^lJ_2(x-y)\boldsymbol{v}(y)\dy,\\[1mm]
 &P(x,y,z)=\big(Q_1(x)-d_1j_1(x)y-ay+H(z), \;Q_2(x)-d_2j_2(x)z-bz+G(y)\big).
 \eess
Clearly, $P(x,y,z)$ is continuous in $\{(x,y,z):0\le x\le l,y\ge0,z\ge0\}$, and $P(x, \boldsymbol{u}(x), \boldsymbol{v}(x))=(0,0)$ for all $0<x<l$.
With regard to $0<x<l$, $y>0$, $z>0$ satisfying $P(x,y,z)=(0,0)$, direct computations yield
 \bess
 \frac{\partial P(x,y,z)}{\partial(y,z)}&=&\begin{pmatrix}                                                  -d_1j_1(x)-a &  H'(z) \\
  G'(y) & -d_2j_2(x)-b
   \end{pmatrix},
 \eess
which is continuous for $0<x<l$, $y,z>0$, and
 \bess
 {\rm det}\frac{\partial P(x,y,z)}{\partial(y,z)}
 =\frac{(Q_1(x)+H(z))(Q_2(x)+G(y))}{yz}-H'(z)G'(y)\ge\frac{H(z)G(y)}{yz}-H'(z)G'(y)>0.
 \eess
Hence, by the implicit function theorem, we know that $(\boldsymbol{u},\boldsymbol{v})$ is continuous in $(0,l)$.

In the following we prove that $(\boldsymbol{u},\boldsymbol{v})$ is continuous at $x=0, l$. We only deal with  $x=0$. Recall that $(\ep\phi_1,\ep\phi_2)\le(\boldsymbol{u},\boldsymbol{v})\le (U,V)$ and $\phi\in X^{++}$. Let $x_n\to 0$ and $(\boldsymbol{u}(x_n),\boldsymbol{v}(x_n))\to (\boldsymbol{u}_0, \boldsymbol{v}_0)$ as $n\to\infty$. Clearly, $\boldsymbol{u}_0$ and $\boldsymbol{v}_0$ are positive. Taking $x=x_n$ in \eqref{2.17} and then letting $n\to\infty$ yield
 \bess\left\{\begin{aligned}
&d_1\int_0^lJ_1(y)\boldsymbol{u}(y)\dy-d_1j_1(0)\boldsymbol{u}_0-a\boldsymbol{u}_0+H(\boldsymbol{v}_0)=0, & & x\in[0,l],\\
&d_2\int_0^lJ_2(y)\boldsymbol{v}(y)\dy-d_2j_2(0)\boldsymbol{v}_0-b\boldsymbol{v}_0+G(\boldsymbol{u}_0)=0, & & x\in[0,l].
 \end{aligned}\right.
 \eess
Then setting $x=0$ in \eqref{2.17}, we can argue as in the proof of Lemma \ref{l2.3} to derive that $(\boldsymbol{u}_0, \boldsymbol{v}_0)=(\boldsymbol{u}(0), \boldsymbol{v}(0))$. Hence, $(\boldsymbol{u},\boldsymbol{v})$ is continuous at $x=0$.

{\it Step 3}: {\it The uniqueness and $(U-\boldsymbol{u},V-\boldsymbol{v})\in X^{++}$}. These two results directly follow from Lemma \ref{l2.3} since \eqref{2.23} and \eqref{2.24} are strict when $x<l$ and near $l$. The details are ignored.

{\it Step 4}: {\it The monotonicity of $(\boldsymbol{u}_l,\boldsymbol{v}_l)$ in $l$ and convergence of $(\boldsymbol{u}_l,\boldsymbol{v}_l)$ as $l\to\yy$.} For any large $l_1>l_2>0$, let $(\boldsymbol{u}_i,\boldsymbol{v}_i)$ be the bounded positive solutions of \eqref{2.17} with $l=l_i$. Then we have
 \bess\left\{\begin{aligned}
&d_1\int_0^{l_2}J_1(x-y)\boldsymbol{u}_1(y)\dy-d_1j_1(x)\boldsymbol{u}_1-a\boldsymbol{u}_1+H(\boldsymbol{v}_1)<0, & & x\in[0,l_2]\\
&d_2\int_0^{l_2}J_2(x-y)\boldsymbol{v}_1(y)\dy-d_2j_2(x)\boldsymbol{v}_1-b\boldsymbol{v}_1+G(\boldsymbol{u}_1)<0, & & x\in[0,l_2].
 \end{aligned}\right.
 \eess
Thus, by Lemma \ref{l2.3}, $(\boldsymbol{u}_1,\boldsymbol{v}_1)>(\boldsymbol{u}_2,\boldsymbol{v}_2)$. That is, $(\boldsymbol{u}_l,\boldsymbol{v}_l)$ is strictly increasing in $l$. Recalling $\boldsymbol{u}_l\le U$ and $\boldsymbol{v}_l\le V$, we have that the limits $\lim_{l\to\yy}\boldsymbol{u}_l(x)=\tilde{\boldsymbol{u}}(x)$ and $\lim_{l\to\yy}\boldsymbol{v}_l(x)=\tilde{\boldsymbol{v}}(x)$ exist for all $x\ge0$ with $0<\tilde{\boldsymbol{u}}\le U$ and $0<\tilde{\boldsymbol{v}}\le V$. The dominated convergence theorem leads to
  \bes\left\{\begin{aligned}
&d_1\int_0^{\yy}J_1(x-y)\tilde{\boldsymbol{u}}(y)\dy-d_1j_1(x)\tilde{\boldsymbol{u}}-a\tilde{\boldsymbol{u}}+H(\tilde{\boldsymbol{v}})=0, & & x\in[0,\yy),\\
&d_2\int_0^{\yy}J_2(x-y)\tilde{\boldsymbol{v}}(y)\dy-d_2j_2(x)\tilde{\boldsymbol{v}}-b\tilde{\boldsymbol{v}}+G(\tilde{\boldsymbol{u}})=0, & & x\in[0,\yy).
 \end{aligned}\right.\lbl{2.25}
 \ees
Then, by the similar lines as in Step 2, we can show that $(\tilde{\boldsymbol{u}},\tilde{\boldsymbol{v}})$ is continuous on $[0,\yy)$.

It will be proved that $(\tilde{\boldsymbol{u}},\tilde{\boldsymbol{v}})=(U,V)$.  Obviously, it is sufficient to show $\inf_{[0,\yy)}\tilde{\boldsymbol{u}}=U$ or $\inf_{[0,\yy)}\tilde{\boldsymbol{v}}=V$ since these two equalities are equivalent. To save space, we denote $\tilde{\boldsymbol{u}}_{\rm inf}=\inf_{[0,\yy)}\tilde{\boldsymbol{u}}$ and $\tilde{\boldsymbol{v}}_{\rm inf}=\inf_{[0,\yy)}\tilde{\boldsymbol{v}}$. We now prove $\tilde{\boldsymbol{u}}_{\rm inf}=U$. Assume on the contrary that $\tilde{\boldsymbol{u}}_{\rm inf}<U$.

{\it Case 1}: $\tilde{\boldsymbol{u}}(x_0)=\tilde{\boldsymbol{u}}_{\rm inf}$ for some $x_0\ge0$. Then
 \bes
 0\le d_1\int_0^{\yy}J_1(x_0-y)\tilde{\boldsymbol{u}}(y)\dy-d_1j_1(x_0)\tilde{\boldsymbol{u}}(x_0)
 =a\tilde{\boldsymbol{u}}(x_0)-H(\tilde{\boldsymbol{v}}(x_0))
 \lbl{2.26}\ees
as $\tilde{\boldsymbol{u}}(y)\ge \tilde{\boldsymbol{u}}_{\rm inf}=\tilde{\boldsymbol{u}}(x_0)$ for all $y\ge 0$. Therefore, $H(\tilde{\boldsymbol{v}}(x_0))\le a\tilde{\boldsymbol{u}}(x_0)<aU=H(V)$.
Since $H(z)$ is strict increasing in $z\ge 0$, it follows that $\tilde{\boldsymbol{v}}(x_0)<V$. So, $\tilde{\boldsymbol{v}}_{\rm inf}<V$.

If $\tilde{\boldsymbol{v}}(x_1)=\tilde{\boldsymbol{v}}_{\rm inf}$ for some $x_1\ge 0$. Similar to the above, we can get  $b\tilde{\boldsymbol{v}}(x_1)-G(\tilde{\boldsymbol{u}}(x_1))\ge0$, which implies $\tilde{\boldsymbol{u}}(x_1)<U$. To sum up, we have
   \bess
a\tilde{\boldsymbol{u}}(x_0)-H(\tilde{\boldsymbol{v}}(x_0))\ge0, ~ ~ b\tilde{\boldsymbol{v}}(x_1)-G(\tilde{\boldsymbol{u}}(x_1))\ge0, ~ ~ \tilde{\boldsymbol{u}}(x_0)\le \tilde{\boldsymbol{u}}(x_1)<U, ~ ~ \tilde{\boldsymbol{v}}(x_1)\le \tilde{\boldsymbol{v}}(x_0)<V.
 \eess
It follows that $G\big({H(\tilde{\boldsymbol{v}}(x_0))}/a\big)\le b\tilde{\boldsymbol{v}}(x_0)$. This contradicts the fact that $(U, V)$ is unique positive root of \qq{1.2}. So $\tilde{\boldsymbol{v}}(x)>\tilde{\boldsymbol{v}}_{\rm inf}$ for all $x\ge 0$.

Then there exists a sequence $\{x_n\}$ with $x_n\nearrow\yy$ such that $\tilde{\boldsymbol{v}}(x_n)\to \tilde{\boldsymbol{v}}_{\rm inf}$ as $n\to\yy$. By passing a subsequence, still denoted by itself, we have $\tilde{\boldsymbol{u}}(x_n)\to \boldsymbol{u}_0$ as $n\to\yy$. Clearly,
 \bes
 \tilde{\boldsymbol{u}}(x_0)=\tilde{\boldsymbol{u}}_{\rm inf}\le \boldsymbol{u}_0\le U,\;\;\;\text{and}\;\; \tilde{\boldsymbol{v}}_{\rm inf}\le \tilde{\boldsymbol{v}}(x_0)<V.\lbl{2.27}\ees
As $\tilde{\boldsymbol{v}}(y)>\tilde{\boldsymbol{v}}_{\rm inf}$ for all $y\ge 0$, it is clear that
 \bes
  \liminf_{n\to\yy}\int_0^{\yy}J_2(x_n-y)\tilde{\boldsymbol{v}}(y)\dy\ge
  \tilde{\boldsymbol{v}}_{\rm inf}\liminf_{n\to\yy}\int_{-x_n}^{\yy}J_2(y)\dy
  =\tilde{\boldsymbol{v}}_{\rm inf},\lbl{2.28}
  \ees
and $j_i(x_n)\to 1$ as $n\to\yy$, $i=1,2$. Together with the equation of $\tilde{\boldsymbol{v}}$, we have $b\tilde{\boldsymbol{v}}_{\rm inf}\ge G(\boldsymbol{u}_0)$. Together with \qq{2.26} and \qq{2.27}, we have
   \bess
a\tilde{\boldsymbol{u}}(x_0)-H(\tilde{\boldsymbol{v}}(x_0))\ge0, ~ ~ b\tilde{\boldsymbol{v}}(x_0)-G(\tilde{\boldsymbol{u}}(x_0))\ge0, ~ ~ \tilde{\boldsymbol{u}}(x_0)\le \boldsymbol{u}_0\le U, ~ ~ \tilde{\boldsymbol{v}}(x_0)<V,
 \eess
which also leads to $G\big({H(\boldsymbol{v}(x_0))}/a\big)\le b\boldsymbol{v}(x_0)$. Analogously, we can get a contradiction.

{\it Case 2}: $\tilde{\boldsymbol{u}}(x)>\tilde{\boldsymbol{u}}_{\rm inf}$ for all $x\ge0$.  If there exists $x_0\ge 0$ such that $\tilde{\boldsymbol{v}}(x_0)=\tilde{\boldsymbol{v}}_{\rm inf}$, by exchanging  the positions of $\tilde{\boldsymbol{u}}$ and $\tilde{\boldsymbol{v}}$, similar to the above (the third paragraph in Case 1) we can derive a contradiction. Therefore, $\tilde{\boldsymbol{u}}(x)>\tilde{\boldsymbol{u}}_{\rm inf}$ and $\tilde{\boldsymbol{v}}(x)>\tilde{\boldsymbol{v}}_{\rm inf}$ for all $x\ge0$. We can find $x_n\nearrow\yy$ and $x_n'\nearrow\yy$ such that $\tilde{\boldsymbol{u}}(x_n)\to \tilde{\boldsymbol{u}}_{\rm inf}$ and $\tilde{\boldsymbol{v}}(x'_n)\to \tilde{\boldsymbol{v}}_{\rm inf}$ as $n\to\infty$. Moreover, by selecting subsequences if necessary, we may assume that $\tilde{\boldsymbol{u}}(x'_n)\to \boldsymbol{u}_0$ and $\tilde{\boldsymbol{v}}(x_n)\to \boldsymbol{v}_0$. Clearly, $\tilde{\boldsymbol{u}}_{\rm inf}\le \boldsymbol{u}_0\le U, ~ ~ \tilde{\boldsymbol{v}}_{\rm inf}\le \boldsymbol{v}_0\le V$.
Taking $x=x_n$ and $x=x'_n$ in the first and second equations of \qq{2.25}, respectively, and then letting $n\to\yy$ we can obtain that, similar to the above (cf. the derivation of \qq{2.28}), $a\tilde{\boldsymbol{u}}_{\rm inf}\ge H(\boldsymbol{v}_0)\ge H(\tilde{\boldsymbol{v}}_{\rm inf})$ and $b\tilde{\boldsymbol{v}}_{\rm inf}\ge G(\boldsymbol{u}_0)\ge G(\tilde{\boldsymbol{u}}_{\rm inf})$.
Note that $\tilde{\boldsymbol{u}}_{\rm inf}<U$ and $\tilde{\boldsymbol{v}}_{\rm inf}<V$. Then  a similar contradiction can be obtained.

The above arguments show that $\tilde{\boldsymbol{u}}_{\rm inf}=U$. Therefore $(\tilde{\boldsymbol{u}},\tilde{\boldsymbol{v}})=(U,V)$.
Then by Dini's theorem, conclusion (1) is obtained.

(2) Since one can prove this assertion by following similar lines as in \cite[Proposition 3.4]{DN8} or \cite[Proposition 2.10]{NV}, we omit the details. The proof is complete.
 \end{proof}

At the end of this section, we show dynamics of the following problem with fixed boundary:
\bes\left\{\begin{aligned}\label{2.29}
&u_t=d_1\int_0^lJ_1(x-y)u(y)\dy-d_1j_1(x)u-au+H(v), & & t>0, ~ ~ x\in[0,l]\\
&v_t=d_2\int_0^lJ_2(x-y)v(y)\dy-d_2j_2(x)v-bv+G(u), & & t>0, ~ ~ x\in[0,l],\\
&u(0,x)=\tilde{u}_0(x), ~ ~v(0,x)=\tilde{v}_0(x),
 \end{aligned}\right.
 \ees
where $(\tilde u_0,\tilde v_0)\in X^+\setminus\{(0,0)\}$. Especially, it will be shown that $(0,0)$ can be exponentially or algebraically stable.

 \begin{lemma}\label{l2.5} Let $(u,v)$ be the unique solution of \eqref{2.29}. Then the following statements are valid.\vspace{-1.5mm}
 \begin{enumerate}[$(1)$]
 \item If $\lambda_1(l)>0$, then $(u(t,x),v(t,x))\to(\boldsymbol{u},\boldsymbol{v})$ in $X$ as $t\to\yy$, where $(\boldsymbol{u},\boldsymbol{v})$ is the unique positive steady state of \qq{2.17}.\vspace{-1.5mm}
 \item If $\lambda_1(l)\le0$, then $(u(t,x),v(t,x))\to(0,0)$ in $X$ as $t\to\yy$. Moreover,
     \begin{enumerate}[$(1)$]
 \item[$(2a)$] if $\lambda_1(l)<0$, then $({\rm e}^{kt}u(t,x),{\rm e}^{kt}v(t,x))\to(0,0)$ in $X$ as $t\to\yy$ for all $k\in(0,-\lambda_1(l))$;
 \item[$(2b)$]if $\lambda_1(l)=0$, and $H, G\in C^2([0,\yy))$ and $H''(z)<0$, $G''(z)<0$ for $z\ge 0$, then there exists a $k_0\in(0,1)$ such that $((t+1)^ku(t,x),(t+1)^kv(t,x))\to(0,0)$ in $X$ as $t\to\yy$ for all $k\in(0,k_0]$.\vspace{-1.5mm}
 \end{enumerate} \end{enumerate}
 \end{lemma}

 \begin{proof}The convergence results in $X$ can be proved by using similar methods as in \cite[Proposition 3.4]{DN8}. Hence we only prove the exponential stability and algebraic stability, respectively, which are obtained by constructing suitable upper solutions.

{\it Exponential stability.} Let $\phi=(\phi_1,\phi_2)$ be the corresponding positive eigenfunction of $\lambda_1(l)$. Define $\bar{u}=M{\rm e}^{-kt}\phi_1(x)$ and $\bar{v}=M{\rm e}^{-kt}\phi_2(x)$
with positive constants $M$ and $k$ to be determined later. We now show that, by choosing suitable $M$ and $k$, $(\bar{u},\bar{v})$ is an upper solution of \eqref{2.29}. Then the desired result follows from a comparison argument.

 Direct computations yield that, for $t>0$ and $x\in[0,l]$,
 \bess
 &&\bar{u}_t-d_1\int_0^lJ_1(x-y)\bar{u}(t,y)\dy+d_1j_1(x)\bar{u}+a\bar{u}-H(\bar{v})\\
 &=&M{\rm e}^{-kt}\left(-k\phi_1-\lambda_1(l)\phi_1+H'(0)\phi_2
 -\frac{H(M{\rm e}^{-kt}\phi_2)}{M{\rm e}^{-kt}}\right)\\
 &\ge& M{\rm e}^{-kt}\left(-k-\lambda_1(l)\right)\phi_1\ge0
 \eess
 provided that $0<k\le -\lambda_1(l)$. Analogously, we can show
 \[\bar{v}_t\ge d_2\int_0^lJ_2(x-y)\bar{v}(t,y)\dy-d_2j_2(x)\bar{v}-b\bar{v}+G(\bar{u})\]
 for $t>0$ and $x\in[0,l]$ if $k\le -\lambda_1(l)$. Moreover, let $M$ large enough such that $\bar{u}(0,x)=M\phi_1(x)\ge \tilde{u}_0(x)$ and $\bar{v}(0,x)=M\phi_2(x)\ge \tilde{v}_0(x)$ for $x\in[0,l]$. Hence $(\bar{u},\bar{v})$ is an upper solution of \eqref{2.29}.

{\it Algebraic stability.} Remember $\lambda_1(l)=0$ in this case, and let $\phi=(\phi_1,\phi_2)$ be the corresponding positive eigenfunction. Fix $M>0$ such that $\bar{u}(0,x)=M\phi_1(x)\ge\tilde{u}_0(x)$ and $\bar{v}(0,x)=M\phi_2(x)\ge\tilde{v}_0(x)$. Let $\bar{u}=M(t+1)^{-k}\phi_1$ and $\bar{v}=M(t+1)^{-k}\phi_2$,
where $k>0$ is chosen later. As above, we only need to show that $(\bar{u},\bar{v})$ is an upper solution of \eqref{2.29}. For clarity, denote $M_i=\max_{[0,l]}\phi_i$ and $m_i=\min_{[0,l]}\phi_i$ with $i=1,2$.

For $t>0$ and $x\in[0,l]$, using the properties of $H$ and the mean value theorem, we have
 \bess
 &&\bar{u}_t-d_1\int_0^lJ_1(x-y)\bar{u}(t,y)\dy+d_1j_1(x)\bar{u}+a\bar{u}-H(\bar{v})\\[1mm]
 &=&\frac{M}{(t+1)^k}\left(\frac{-k\phi_1}{t+1}+H'(0)\phi_2-
 \frac{H(M(t+1)^{-k}\phi_2)}{M(t+1)^{-k}}\right)\\[1mm]
 &=&-\frac{M}{(t+1)^k}\left(\frac{k\phi_1}{t+1}+\frac{
 M\phi^2_2}{2(t+1)^k}H''(\xi)\right) ~ ~ \big( \text{here}\;\; \xi\in(0,M(t+1)^{-k}\phi_2)\big)\\[1mm]
 &\ge&-\frac{M}{(t+1)^k}\left(\frac{kM_1}{t+1}+\frac{Mm^2_2}{2(t+1)^k}\dd\max_{[0,\, MM_2]} H''\right)=-\frac{M}{(t+1)^{2k}}\left(\frac{kM_1}{(t+1)^{1-k}}
 +\frac{Mm^2_2}2\dd\max_{[0,\, MM_2] }H''\right)\\[1mm]
 &\ge&-\frac{M}{(t+1)^{2k}}\left(kM_1+\frac {Mm^2_2}2\dd\max_{[0,\, MM_2]} H''\right)\ge0
 \eess
if $0<k\le\min\kk\{1, \; -\frac {Mm^2_2}{2M_1}\dd\max_{[0,\, MM_2]}H''\rr\}$.
Similarly, we can get that, for $t>0$ and $x\in[0,l]$,
 \[\bar{v}_t\ge d_2\int_0^lJ_2(x-y)\bar{v}(t,y)\dy-d_2j_2(x)\bar{v}-b\bar{v}+G(\bar{u})\]
when $0<k\le\min\kk\{1,\; -\frac{Mm^2_1}{2M_2}\dd\max_{[0,\,MM_1]}G''\rr\}$.
This completes the proof.
 \end{proof}

\section{Dynamics of \eqref{1.10}}\lbl{s4}

In this section, we investigate the dynamics of \eqref{1.10}. We first show spreading-vanishing dichotomy holds, and then discuss the criteria governing spreading and vanishing.

\subsection{Spreading-vanishing dichotomy and long time behaviors}

The following theorem shows that similar to \eqref{1.9} (see \cite[Theorem 1.1]{NV}), the dynamics of \eqref{1.10} also conforms to a spreading-vanishing dichotomy. Besides we prove that when vanishing happens, $(0,0)$ can be exponentially or algebraically asymptotically stable, depending on the sign of a related principal eigenvalue.

\begin{theorem}[Spreading-vanishing dichotomy]\label{t1.2} Let $(u,v,h)$ be the unique solution of \eqref{1.10}. Then one of the following alternatives must happen.\vspace{-1.5mm}
\begin{enumerate}[$(1)$]
\item \underline{Spreading {\rm(}necessarily $\mathcal{R}_0>1${\rm)}}: $ h_\yy:=\lim_{t\to\yy}h(t)=\yy$, $\lim_{t\to\yy}u(t,x)=U$ and $\lim_{t\to\yy}v(t,x)=V$ in $C_{\rm loc}([0,\yy))$, where $(U,V)$ is uniquely given by \eqref{1.2}.\vspace{-1.5mm}
\item \underline{Vanishing:} $h_{\yy}<\yy$, $\lambda_1(h_{\yy})\le0$ and $\lim_{t\to\yy}\|u(t,\cdot)+v(t,\cdot)\|_{C([0,h(t)])}=0$, where $\lambda_1(h_{\yy})$ is the principal eigenvalue of \eqref{2.18}. Moreover, \vspace{-1.5mm}
\begin{enumerate}
\item[$(1a)$] if $\lambda_1(h_{\yy})<0$, then $\lim_{t\to\yy}{\rm e}^{kt}\|u(t,\cdot)+v(t,\cdot)\|_{C([0,h(t)])}=0$ for any $k\in(0,-\lambda_1(h_{\yy}))$;
\item[$(1b)$] if $\lambda_1(h_{\yy})=0$, there exists a small $k_0>0$ such that $\lim_{t\to\yy}(1+t)^{k}\|u(t,\cdot)+v(t,\cdot)\|_{C([0,h(t)])}=0$ for any $k\in(0,k_0)$.\vspace{-1.5mm}
    \end{enumerate}
    \end{enumerate}
\end{theorem}

Theorem \ref{t1.2} can be obtained by the following two lemmas.

\begin{lemma}\label{l3.1}If $h_{\yy}<\yy$, then $\lambda_1(h_{\yy})\le0$ and $\lim_{t\to\yy}\|u(t,x)+v(t,x)\|_{C([0,h(t)])}=0$. Moreover,
\vspace{-1.5mm}
 \begin{enumerate}[$(1)$]
 \item if $\lambda_1(h_{\yy})<0$, then $\lim_{t\to\yy}{\rm e}^{kt}\|u(t,x)+v(t,x)\|_{C([0,h(t)])}=0$ for all $0<k<-\lambda_1(h_{\yy})$;\vspace{-1.5mm}
 \item if $\lambda_1(h_{\yy})=0$, and $H, G\in C^2([0,\yy))$ and $H''(z)<0$, $G''(z)<0$ for $z\ge 0$, then there exists a $k_0\in(0,1]$ such that $\lim_{t\to\yy}(t+1)^k\|u(t,x)+v(t,x)\|_{C([0,h(t)])}=0$ for all $0<k\le k_0$.\vspace{-1.5mm}
     \end{enumerate}
\end{lemma}

\begin{proof}We first prove that if $h_{\yy}<\yy$, then $\lambda_1(h_{\yy})\le0$. Assume on the contrary that $\lambda_1(h_{\yy})>0$. By the continuity of $\lambda_1(l)$ in $l$, there exist small $\ep>0$ and $\delta>0$ such that $\lambda_1(h_{\yy}-\ep)>0$ and $\min\{J_1(x),J_2(x)\}\ge\delta$ for $|x|\le 2\ep$ due to the condition {\bf (J)}.  Moreover, there is $T>0$ such that $h(t)>h_{\yy}-\ep$ for $t\ge T$. Hence the solution component $(u,v)$ of \eqref{1.10} satisfies
\bess\left\{\begin{aligned}
&u_t\ge d_1\int_0^{h_{\yy}-\ep}J_1(x-y)u(y)\dy-d_1j_1(x)u-au+H(v), & & t>T, ~ ~ x\in[0,h_{\yy}-\ep]\\
&v_t\ge d_2\int_0^{h_{\yy}-\ep}J_2(x-y)v(y)\dy-d_2j_2(x)v-bv+G(u), & & t>T, ~ ~ x\in[0,h_{\yy}-\ep],\\
&u(T,x)>0, ~ ~v(T,x)>0, & & x\in[0,h_{\yy}-\ep].
 \end{aligned}\right.
 \eess
Let $(\underline{u},\underline{v})$ be the unique solution of \eqref{2.29} with $l=h_{\yy}-\ep$, $\tilde{u}_0(x)=u(T,x)$ and $\tilde{v}_0(x)=v(T,x)$. Note that $\lambda_1(h_{\yy}-\ep)>0$. Making use of Lemma \ref{l2.5} we have $(\underline{u},\underline{v})\to(\boldsymbol{u},\boldsymbol{v})$ in $X$ as $t\to\yy$, where $(\boldsymbol{u},\boldsymbol{v})$ is the unique positive solution of \eqref{2.17} with $l=h_{\yy}-\ep$. Furthermore, by comparison principle, $u(t+T,x)\ge \underline{u}(t,x)$ and $v(t+T,x)\ge \underline{v}(t,x)$ for $x\in[0,h_{\yy}-\ep]$. Therefore, $\liminf_{t\to\yy}(u(t,x),v(t,x))\ge(\boldsymbol{u},\boldsymbol{v})$ uniformly in $[0,h_{\yy}-\ep]$. There exist small $\sigma>0$ and large $T_1\gg T$ such that $u(t,x)\ge\sigma$ and $v(t,x)\ge\sigma$ for $t\ge T_1$ and $[0,h_{\yy}-\ep]$. In view of the equation of $h(t)$, we have, for $t>T_1$,
 \bess
 h'(t)\ge\sigma\int_{h_{\yy}-\frac{3\ep}2}^{h_{\yy}-\ep}\int_{h_{\yy}}^{h_{\yy}
 +\frac{\ep}2}\big[\mu_1J_1(x-y)+\mu_2J_2(x-y)\big]\dy\dx\ge(\mu_1+\mu_2)\delta\sigma,
 \eess
which clearly contradicts $h_{\yy}<\yy$. Thus $\lambda_1(h_{\yy})\le0$.

 Let $(\bar{u},\bar{v})$ be the solution of \eqref{2.21} with $l=h_{\yy}$, $\tilde{u}_0(x)=\|u_0\|_{C([0,h_0])}$ and $\tilde{v}_0(x)=\|v_0\|_{C([0,h_0])}$. Clearly, $\bar{u}(t,x)\ge u(t,x)$ and $\bar{v}\ge v(t,x)$ for $t\ge0$ and $x\in[0,h(t)]$. Note that $\lambda_1(h_{\yy})\le0$. Then the convergence results in this lemma follow from Lemma \ref{l2.5}. The proof is ended.
\end{proof}
The proof of the following result is standard, so the details are omitted.
\begin{lemma}\label{l3.2} If $h_{\yy}=\yy$ {\rm(}necessarily $\mathcal{R}_0>1$, see Lemma \ref{l3.3}{\rm)}, then $(u(t,x),v(t,x))\to(U,V)$ in $C_{\rm loc}([0,\yy))$ as $t\to\yy$.
\end{lemma}

\subsection{The criteria for spreading and vanishing}

We shall give a rather complete description of criteria for spreading and vanishing. From this result, one can learn some effect, brought by the cooperative behaviors of two agents $u$ and $v$, on spreading and vanishing. Define
  \[\mathcal{R}_*=\mathcal{R}_*(d_1, d_2):=\frac{H'(0)G'(0)}{(a+\frac{d_1}2)(b+\frac{d_2}2)}.\]
The main conclusion of this subsection is the following theorem.

\begin{theorem}[Criteria for spreading and vanishing]\label{t1.3} Let $\mathcal{R}_0$ be given by \eqref{x.1}, and $(u,v,h)$ be the unique solution of \eqref{1.10}. Then the following results hold.\vspace{-1.5mm}
\begin{enumerate}[$(1)$]
\item If $\mathcal{R}_0\le1$, then vanishing happens.\vspace{-1.5mm}
\item If $\mathcal{R}_*\ge1$, then spreading occurs.
  \vspace{-1.5mm}
\item Assume $\mathcal{R}_*<1<\mathcal{R}_0$ and fix all parameters but except for $h_0$ and $\mu_i$ for $i=1,2$. We can find a unique $\ell^*>0$ such that\vspace{-1.5mm}
\begin{enumerate}
\item[{\rm (3a)}] if $h_0\ge\ell^*$, then spreading happens;
\item[{\rm (3b)}] if $h_0<\ell^*$, then the following statements hold:
\begin{enumerate}\vspace{-1.5mm}
\item[{\rm (3b$_1$)}] there exists $\underline{\mu}>0$ such that vanishing happens when $\mu_1+\mu_2\le\underline{\mu}$; and there exists a $\bar{\mu}_1>0$ $(\bar{\mu}_2>0)$ which is independent of $\mu_2$\,$(\mu_1)$ such that spreading happens when  $\mu_1\ge\bar{\mu}_1$\,$(\mu_2\ge\bar{\mu}_2)$;
\item[{\rm (3b$_2$)}] if $\mu_2=f(\mu_1)$ where $f\in C([0,\yy))$, is strictly increasing, $f(0)=0$ and $\lim\limits_{s\to\yy}f(s)=\yy$, then there exists a unique $\mu^*_1>0$ such that spreading happens if and only if $\mu_1>\mu^*_1$. \vspace{-1.5mm}
    \end{enumerate}
    \end{enumerate}\vspace{-1.5mm}
\item Assume $\mathcal{R}_*<1<\mathcal{R}_0$ and fix all parameters but except for $d_i$ and $\mu_i$, $i=1,2$. \vspace{-1.5mm}
\begin{enumerate}
\item[{\rm (4a)}] Let $d_2=f(d_1)$ with $f$ having the properties as in {\rm (3b$_2$)}, and $\ud d_1>0$ be the unique root of $\mathcal{R}_*(d_1, f(d_1))=1$\, $(\mathcal{R}_*(d_1, f(d_1))<1$ is equivalent to $d_1>\ud d_1)$. Then there exists a unique ${d}^*_1>\ud d_1$ such that spreading happens if $\ud d_1<d_1\le {d}^*_1$; while if $d_1>{d}^*_1$, then whether spreading or vanishing happens depends on the expanding rates $\mu_1$ and $\mu_2$ as in {\rm (3b$_1$)}.
\item[{\rm (4b)}]
\begin{enumerate}
\item[{\rm(i)}] Fix $d_2<\Lambda:=2(H'(0)G'(0)-ab)/a$ and let $D_1=D_1(d_2)>0$ be the unique root of $\mathcal{R}_*(d_1, d_2)=1$\, $(\mathcal{R}_*(d_1, d_2)<1$ is equivalent to $d_1>D_1)$. Then there exists a unique $\hat{d}_1>D_1$ such that spreading happens if $D_1<d_1\le \hat{d}_1$, while if $d_1>\hat{d}_1$, then whether spreading or vanishing happens depends on the expanding rates $\mu_1$ and $\mu_2$ as in Lemmas \ref{l3.6} and \ref{l3.7};
\item[{\rm(ii)}] Let $\nu(d_2)$ be given by the following \qq{3.1} and $\ud{d}_2\ge\Lambda$ be the unique root of $\nu(d_2)=0$. If we fix $d_2\in[\Lambda, \ud{d}_2)$, then there exists a unique $\tilde d_1>0$ such that spreading happens when $d_1\le\tilde d_1$, while when $d_1>\tilde d_1$, whether spreading or vanishing happens depends on the expanding rates $\mu_1$ and $\mu_2$ as in {\rm (3b$_1$)};
\item[{\rm(iii)}] If we fix $d_2>\ud{d}_2$, then for all $d_1>0$, whether spreading or vanishing happens depends on the expanding rates $\mu_1$ and $\mu_2$ as in {\rm (3b$_1$)}.\vspace{-1.5mm}
\end{enumerate}\end{enumerate}
\end{enumerate}
\end{theorem}

The proof of Theorem \ref{t1.3} will be divided into several lemmas. We start with considering the case $\mathcal{R}_0=H'(0)G'(0)/(ab)\le 1$.

\begin{lemma}\label{l3.3} If $\mathcal{R}_0\le1$, then vanishing happens. Particularly,
 \bes
 h_{\yy}\le h_0+\frac 1{\min\kk\{d_1/\mu_1,\, H'(0)d_2/(b\mu_2)\rr\}}\int_0^{h_0}\left(u_0(x)
+\frac{H'(0)}{b}v_0(x)\right)\dx.\lbl{4.1a}\ees
\end{lemma}

\begin{proof} Firstly, in view of $j_i(x)=\int_0^\yy\!J_1(x-y)\dy$, it can be deduced that
 \bess
 \int_0^{h(t)}\!\!\int_0^{h(t)}\!J_1(x-y)u(t,y)\dy\dx-\int_0^{h(t)}\!j_1(x)u(t,x)\dx
&=&-\int_0^{h(t)}\!\!\int_{h(t)}^\yy\!J_1(x-y)u(t,x)\dx\dy,\\
\int_0^{h(t)}\!\!\int_0^{h(t)}\!J_2(x-y)v(t,y)\dy\dx-\int_0^{h(t)}\!j_2(x)v(t,x)\dx
&=&-\int_0^{h(t)}\!\!\int_{h(t)}^\yy\!J_2(x-y)v(t,x)\dx\dy.
 \eess
Then, by a series of simple computations, we have
\bess
 \frac{\rm d}{{\rm d}t}\!\int_0^{h(t)}\!\!\left(u+\frac{H'(0)}{b}v\right)\dx
&=&-\int_0^{h(t)}\!\!\!\int_{h(t)}^{\yy}\!\!\kk(d_1J_1(x-y)u+\frac{H'(0)d_2}{b}
J_2(x-y)v\!\rr)\!\dy\dx\\[1mm]
&&+\int_0^{h(t)}\kk(H(v)-au-H'(0)v+\frac{H'(0)}{b}G(u)\rr)\dx\\[1mm]
&<&-\min\kk\{d_1/\mu_1,\, H'(0)d_2/(b\mu_2)\rr\}h'(t).
\eess
Hence we derive
\[\frac{\rm d}{{\rm d}t}\int_0^{h(t)}\left(u+\frac{H'(0)}{b}v\right)\dx
<-\min\kk\{d_1/\mu_1,\, H'(0)d_2/(b\mu_2)\rr\}h'(t).\]
Integrating the above inequality from $0$ to $t$ yields \qq{4.1a}.
\end{proof}

The following involves the case $\mathcal{R}_0>1$. All arguments used below tightly depend on the fact that if vanishing happens, then $\lambda_1(h_{\yy})\le0$ as in Lemma \ref{l3.1}. Here we mention that, at our present situation, $a_{11}=-a, a_{22}=-b, a_{12}=H'(0)$ and $a_{21}=G'(0)$. Thus
 \bess
&&\gamma_A=\frac{-(a+b)+\sqrt{(a+b)^2+4\big[H'(0)G'(0)-ab\big]}}2>0,\\
&& \gamma_B=\frac{-(a+\frac{d_1}2+b+\frac{d_2}2)
+\sqrt{(a+\frac{d_1}2+b+\frac{d_2}2)^2+4\big[H'(0)G'(0)-(a+\frac{d_1}2)
 (b+\frac{d_2}2)\big]}}2. \eess
It is clear that $\mathcal{R}_*(d_1, d_2)\ge1$ if and only if $\gamma_B\ge0$.

\begin{lemma}\label{l3.4}If $\mathcal{R}_*\ge1$, then spreading occurs.
\end{lemma}

\begin{proof} The condition $\mathcal{R}_*\ge1$ implies $\gamma_B\ge 0$.
Owing to Proposition \ref{p2.2}(3), $\lim_{l\to 0}\lambda_1(l)=\gamma_B$. Therefore, $\lambda_1(l)>\gamma_B\ge 0$ for all $l>0$. It then follows from Lemma \ref{l3.1} that spreading happens.
\end{proof}

In what follows, we focus on the case $\mathcal{R}_*<1<\mathcal{R}_0$.
We fix all the parameters in \eqref{1.10} but except for $h_0$ and $\mu_i$ with $i=1,2$, and discuss the effect of initial habitat $[0,h_0]$ on criteria of spreading and vanishing. Making use of Proposition \ref{p2.2}, we have $\lim_{l\to\yy}\lambda_1(l)=\gamma_A>0$, and $\lim_{l\to 0}\lambda_1(l)=\gamma_B <0$. By the monotonicity of $\lambda_1(l)$, there exists a unique $\ell^*>0$ such that $\lambda_1(\ell^*)=0$ and $\lambda_1(l)(l-\ell^*)>0$ for $l\neq\ell^*$. As $\lambda_1(l)$ is strictly increasing in $l>0$, as a consequence of  Lemma \ref{l3.1}, we have the following result.

\begin{lemma}\label{l3.5}Let $\ell^*$ be defined as above. If $h_0\ge\ell^*$, then spreading happens.
\end{lemma}

The next result shows that if $h_0<\ell^*$ and $\mu_1+\mu_2$ small enough, then vanishing occurs.

\begin{lemma}\label{l3.6}If $h_0<\ell^*$, then there exists a $\ud\mu>0$ such that vanishing happens if $\mu_1+\mu_2\le\ud\mu$.
\end{lemma}

\begin{proof} Due to $h_0<\ell^*$, we have $\lambda_1(h_0)<0$. By the continuity of $\lambda_1(l)$ (Proposition \ref{p2.2}), there exists a small $\ep>0$ such that $\lambda_1(h_0(1+\ep))<0$. For convenience, denote $h_1=h_0(1+\ep)$. Let $\phi=(\phi_1,\phi_2)$ be the positive eigenfunction of $\lambda_1(h_0(1+\ep))$ with $\|\phi\|_{X}=1$. Define $\bar{h}(t)=h_0\big[1+\ep(1-{\rm e}^{-\delta t})\big]$, $\bar{u}(t,x)=M{\rm e}^{-\delta t}\phi_1$ and $\bar{v}=M{\rm e}^{-\delta t}\phi_2$
with $0<\delta\le -\lambda_1(h_1)$ and $M$ large enough such that $M\phi_1(x)\ge u_0(x)$ and $M\phi_2(x)\ge v_0(x)$ for $x\in[0,h_1]$.
Direct calculations yield that, for $t>0$ and $x\in[0,\bar{h}(t)]$,
 \bess
 &&\bar{u}_t-d_1\int_0^{\bar{h}(t)}\!\!J_1(x-y)\bar{u}(t,y)\dy
+d_1j_1(x)\bar{u}+a\bar{u}-H(\bar{v})\\
 &\ge& M{\rm e}^{-\delta t}\left(-\delta\phi_1-d_1\int_0^{h_1}J_1(x-y)\phi_1(y)\dy+d_1j_1(x)\phi_1+a\phi_1-\frac{H(\bar{v})}{M{\rm e}^{-\delta t}}\right)\\
 &=&M{\rm e}^{-\delta t}\left(-\delta\phi_1-\lambda_1(h_1)\phi_1+H'(0)\phi_2-\frac{H(M{\rm e}^{-\delta t}\phi_2)}{M{\rm e}^{-\delta t}}\right)\\
 &\ge& M{\rm e}^{-\delta t}\left(-\delta-\lambda_1(h_1)\right)\phi_1\ge 0.
\eess
Similarly, there holds:
\[\bar{v}_t-d_2\int_0^{\bar{h}}J_2(x-y)\bar{v}(t,y)\dy
+d_2j_2(x)\bar{v}+b\bar{v}-G(\bar{u})\ge0.\]
Moreover, when $\mu_1+\mu_2\le\frac{\ep\delta h_0}{Mh_1}$, we have
 \bess
&&\int_0^{\bar{h}(t)}\!\!\int_{\bar{h}(t)}^{\yy}\big[\mu_1J_1(x-y)\bar{u}(t,x)
+\mu_2J_2(x-y)\bar{v}(t,x)\big]\dy\dx\\
&=& M{\rm e}^{-\delta t}\sum_{i=1}^2\mu_i\int_0^{\bar{h}(t)}\!\!\int_{\bar{h}(t)}^{\yy}J_i(x-y)\phi_i(x)\dy\dx\le(\mu_1+\mu_2)Mh_1{\rm e}^{-\delta t}\le\ep\delta h_0{\rm e}^{-\delta t}=\bar{h}'(t).
  \eess

By the comparison principle, $h(t)\le\bar{h}(t)$ for $t\ge0$, which implies $\lim_{t\to\yy}h(t)<\yy$.
\end{proof}

\begin{lemma}\label{l3.7}If $h_0<\ell^*$, then there exists a $\bar{\mu}_1>0$\, $(\bar{\mu}_2>0)$ which is independent of $\mu_2$ $(\mu_1)$ such that spreading happens when $\mu_1\ge\bar{\mu}_1$ $(\mu_2\ge\bar{\mu}_2)$.
\end{lemma}

\begin{proof}We only prove the assertion about $\mu_1$ since the similar method can be adopt for the conclusion of $\mu_2$.
Let $(\underline{u},\underline{v},\underline{h})$ be the unique solution of \eqref{1.10} with $\mu_2=0$. Clearly, $(\underline{u},\underline{v},\underline{h})$ is an lower solution of \eqref{1.10} and Lemmas \ref{l3.1}-\ref{l3.6} hold for $(\underline{u},\underline{v},\underline{h})$. Then we can argue as in the proof of \cite[Theorem 1.3]{DN8} to deduce that there exists a $\ol{\mu}_1>0$ such that if $\mu_1\ge\ol{\mu}_1$, spreading happens for $(\underline{u},\underline{v},\underline{h})$ and also for the unique solution $(u,v,h)$ of \eqref{1.10}. The proof is finished.
\end{proof}

By Lemmas \ref{l3.6} and \ref{l3.7}, we have that vanishing occurs if $\mu_1+\mu_2\le \ud{\mu}$, while spreading happens if $\mu_1+\mu_2\ge\ol{\mu}_1+\ol{\mu}_2=:\ol\mu$. One naturally wonders whether these is a unique critical value of $\mu_1+\mu_2$ such that spreading happens if and only if $\mu_1+\mu_2$ is beyond this critical value. Indeed, such value does not exist since the unique solution $(u,v,h)$ of \eqref{1.10} is not monotone about $\mu_1+\mu_2$. However, for some special $(\mu_1, \mu_2)$ we can obtain a unique critical value as we wanted.

\begin{lemma}\label{l3.8} Assume $h_0<\ell^*$. If $\mu_2=f(\mu_1)$ where $f\in C([0,\yy))$, is strictly increasing, $f(0)=0$ and $\lim\limits_{s\to\yy}f(s)=\yy$. Then there is a unique $\mu^*_1>0$ such that spreading occurs if and only if $\mu_1>\mu^*_1$.
\end{lemma}

\begin{proof} Firstly, it is easy to see from a comparison argument that the unique solution $(u,v,h)$ is strictly increasing in $\mu_1$. We have known that vanishing happens when $\mu_1+f(\mu_1)\le\ud{\mu}$ (Lemma \ref{l3.6}), and spreading happens when $\mu_1+f(\mu_1)\ge\ol{\mu}$ (Lemma \ref{l3.7}). Due to the properties of $f$, there exist unique $\ud\mu_1$ and $\ol\mu_1>0$,  such that $\ud\mu_1+f(\ud\mu_1)=\ud{\mu}$ and $\ol\mu_1+f(\ol\mu_1)=\ol{\mu}$. Clearly, $\mu_1+f(\mu_1)\le\ud{\mu}$ is equivalent to $\mu_1\le\ud{\mu}_1$, and $\mu_1+f(\mu_1)\ge\ol{\mu}$ is equivalent to $\mu_1\ge\ol{\mu}_1$. So, vanishing happens if $\mu_1\le\ud{\mu}_1$, while spreading occurs if $\mu_1\ge\ol{\mu}_1$. Then we can use the monotonicity of $(u,v,h)$ on $\mu_1$ and argue as in the proof of \cite[Theorem 3.14]{CDLL} to finish the proof. The details are omitted here.
\end{proof}

Next, to investigate the effect of $d_i$ on spreading and vanishing, we fix all the parameters but except for $d_i$ and $\mu_i$ with $i=1,2$. Assume that $d_2=f(d_1)$ with $f$ being given as in Lemma \ref{l3.8}. Then we try to obtain a critical value for $d_1$ governing spreading and vanishing.

Clearly, $\mathcal{R}_*(d_1, f(d_1))$ is strict decreasing in $d_1$. There exists a unique $\ud d_1>0$ such that $\mathcal{R}_*(\ud d_1, f(\ud d_1))=1$, and $[\mathcal{R}_*(d_1, f(d_1))-1](d_1-\ud d_1)<0$ for $d\not=\ud d_1$.

\begin{lemma}\label{l3.9}Suppose that $d_2=f(d_1)$ and $d_1>\ud d_1$. Then there exists a unique ${d}^*_1>\ud d_1$ such that spreading happens if $d_1\le {d}^*_1$, while if $d_1>{d}^*_1$, then whether spreading or vanishing happens depends on the expanding rates $\mu_1$ and $\mu_2$ as in Lemmas \ref{l3.6} and \ref{l3.7}.
\end{lemma}

\begin{proof}
For clarity, we rewrite $\gamma_B$ defined by \qq{2.7} as $\gamma_B(d_1)$, and the principal eigenvalues of \qq{2.18} and \qq{2.19} as $\lambda_1(l, d_1)$ and $\lambda_2(l, d_1)$, respectively. Then $\lambda_2(h_0, d_1)$ is strictly decreasing in $d_1>0$ by Proposition \ref{p2.3}(3).

Note that $\mathcal{R}_*(d_1, f(d_1))\ge1$ is equivalent to $\gamma_B(d_1)\ge0$. Therefore, $0=\gamma_B(\ud d_1)=\dd\lim_{l\to 0}\lambda_1(l, \ud d_1)$, and then $\lambda_1(l, \ud d_1)>0$ for all $l>0$ by the monotonicity of $\lambda_1(l)$. Thanks to \qq{2.20}, $\lambda_2(l, \ud d_1)>0$ for all $l>0$. Certainly, $\lambda_2(h_0, \ud d_1)>0$. Moreover, by Proposition \ref{p2.3}(5), $\lambda_2(h_0, d_1)<0$ when $d_1$ is large. The monotonicity of $\lambda_2(h_0, d_1)$ indicates that there is a unique ${d}^*_1>\ud d_1$ such that $\lambda_2(h_0, d^*_1)=0$ and $\lambda_2(h_0, d_1)(d_1-{d}^*_1)<0$ when $d_1>\ud d_1$ and $d_1\not={d}^*_1$. Recalling \qq{2.20}, it follows that if $\ud{d}_1<d_1\le d^*_1$, then $\lambda_1(h_0, d_1)\ge0$ and spreading happens by Lemma \ref{l3.1}; if $d_1>d^*_1$, then $\lambda_1(h_0, d_1)<0$ and similar to the arguments in the proofs of Lemmas \ref{l3.6} and \ref{l3.7}, we can get the desired results. The proof is finished.
\end{proof}

Next we consider the case where one diffusion coefficient is fixed and the other one varies. Since the situations are parallel, we only study the case where $d_2$ is fixed and $d_1$ is varying. Notice that $\mathcal{R}_*(d_1, d_2)<1<\mathcal{R}_0$. Define $\Lambda=2(H'(0)G'(0)-ab)/a$.
If $d_2<\Lambda$, then $\mathcal{R}_*(0, d_2)>1$. The condition $\mathcal{R}_*(d_1, d_2)<1$ is equivalent to $d_1>D_1$, where $D_1>0$ is the unique root of $\mathcal{R}_*(d_1, d_2)=1$. This leads to the following conclusion. The proof is ignored since it is similar to that of Lemma \ref{l3.9}.

\begin{lemma}\label{l3.9a} Fix $d_2<\Lambda$ and let $d_1>D_1$. Then there exists a unique $\hat{d}_1>D_1$ such that spreading happens if $d_1\le \hat{d}_1$, while if $d_1>\hat{d}_1$, then whether spreading or vanishing happens depends on the expanding rates $\mu_1$ and $\mu_2$ as in Lemmas \ref{l3.6} and \ref{l3.7}.
\end{lemma}

Next we deal with the case $d_2\ge \Lambda$. In view of Proposition \ref{p2.3}, we have $\lambda_1(h_0, d_1)<0$ when $d_1$ is large enough,  and thus $\lambda_2(h_0, d_1)<0$. Using the continuity of $\lambda_1(h_0, d_1)$ in $d_1\ge 0$ (Proposition \ref{p2.3}(1)),  we get $\lambda_1(h_0, d_1)\to\nu_1(d_2)$ as $d_1\to0$, where $\nu_1(d_2)$ is the principal eigenvalue of
 \bess\left\{\begin{aligned}
&-a\phi_1+H'(0)\phi_2=\nu\phi_1, & &x\in[0,h_0],\\[1mm]
&d_2\int_0^{h_0}J_2(x-y)\phi_2(y)\dy-d_2j_2(x)\phi_2+G'(0)\phi_1-b\phi_2=\nu\phi_2,& &x\in[0,h_0].
 \end{aligned}\right.
 \eess
 Let $\kappa_1$ be the principal eigenvalue of
 \[\int_0^{h_0}J_2(x-y)\omega(y)\dy-j_2(x)\omega(x)=\kappa\omega(x) ~ ~ {\rm for ~ }x\in[0,h_0].\]
Then $-1/2<\kappa_1<0$ (cf. \cite[Lemma 2.6]{LLW}). The simple calculations yield
 \bes
 \nu_1(d_2)=\frac{-(a+b-d_2\kappa_1)+\sqrt{(a+b-d_2\kappa_1)^2
  -4\big[a(b-d_2\kappa_1)-H'(0)G'(0)\big]}}2.
 \lbl{3.1}\ees
Clearly, $\nu_1(d_2)$ is strictly decreasing in $d_2$ and $\nu_1(\Lambda)>0$. Since
$\nu_1(d_2)\to-a$ as $d_2\to\yy$, there exists a unique $\ud{d}_2>\Lambda$ such that $\nu_1(\ud{d}_2)=0$ and $\nu_1(d_2)(d_2-\ud{d}_2)<0$ for $\Lambda<d_2\neq \ud{d}_2$.

\begin{lemma}\label{l3.10}The following statements are valid.\vspace{-1.5mm}
\begin{enumerate}[$(1)$]
\item If we fix $d_2\in[\Lambda,\, \ud{d}_2)$, then there exists a unique $\tilde{d}_1>0$ such that spreading happens if $d_1\le\tilde{d}_1$, while whether spreading or vanishing happens depends on the expanding rates $\mu_1$ and $\mu_2$ as in Lemmas \ref{l3.6} and \ref{l3.7} if $d_1>\tilde{d}_1$.\vspace{-1.5mm}
\item If we fix $d_2\ge \ud{d}_2$, then for all $d_1>0$, whether spreading or vanishing happens depends on the expanding rates $\mu_1$ and $\mu_2$ as in Lemmas \ref{l3.6} and \ref{l3.7}.\vspace{-1.5mm}
\end{enumerate}
 \end{lemma}

\begin{proof} (1) Since $d_2\in[\Lambda,\, \ud{d}_2)$, we have $\nu_1(d_2)>0$. Recalling that $\lambda_1(h_0, d_1)$ and $\lambda_2(h_0, d_1)$ have the same sign. From the above discussion we see that $\lambda_2(h_0, d_1)>0$ if $d_1\ll 1$, while $\lambda_2(h_0, d_1)<0$ if $d_1\gg 1$. By the monotonicity of $\lambda_2(h_0, d_1)$ on $d_1$, then there exists a unique $\tilde d_1>0$ such that $\lambda_2(h_0, d_1)=0$ if $d_1=\tilde d_1$ and $\lambda_2(h_0, d_1)(d_1-\tilde d_1)<0$ if $d_1\neq \tilde d_1$. Similar to the proof of Lemma \ref{l3.9}, the first assertion is proved.

(2) If $d_2\ge \ud{d}_2$, then $\nu_1(d_2)\le0$. We claim that $\lambda_1(h_0, d_1)<0$ for all $d_1>0$. Otherwise, then $\lambda_1(h_0, d_1)\ge0$ for some $d_1>0$,  and so $\lambda_2(h_0, d_1)\ge0$. By the monotonicity of $\lambda_2(h_0, d_1)$ in $d_1$, $\lambda_2(h_0, d_1)>0$ when $d_1\ll 1$, which leads to $\lambda_1(h_0, d_1)>0$ when $d_1\ll 1$. Thus, $\nu_1(d_2)=\dd\lim_{d_1\to0}\lambda_1(h_0, d_1)\ge 0$. Recall $\nu_1(d_2)\le0$. So $\lim_{d_1\to0}\lambda_1(h_0, d_1)=0$, which implies $\lim_{d_1\to0}\lambda_2(h_0, d_1)=0$. However, as $\lambda_2(h_0, d_1)>0$ for $0<d_1\ll 1$, it follows from the monotonicity that $\lim_{d_1\to0}\lambda_2(h_0, d_1)>0$. This is a contradiction.

Therefore, $\lambda_1(h_0, d_1)<0$ for all $d_1>0$. Then by arguing as in the proof of Lemma \ref{l3.9}, we can complete the proof whose details are ignored. Therefore the proof is finished.
 \end{proof}

Theorem \ref{t1.3}(1) and (2) are exactly Lemmas \ref{l3.3}-\ref{l3.4}, respectively; Theorem \ref{t1.3}($3a$) is exactly Lemma \ref{l3.5}; Theorem \ref{t1.3}($3b$) follows from Lemmas \ref{l3.6}-\ref{l3.8}; Theorem \ref{t1.3}(4) follows from Lemmas \ref{l3.9}-\ref{l3.10}.

\section{Spreading speed}\lbl{s5}
\setcounter{equation}{0} {\setlength\arraycolsep{2pt}

In this section, we investigate the spreading speed of \eqref{1.10}, and thus always assume that spreading occurs for \eqref{1.10} which implies $\mathcal{R}_0>1$. It will be seen that accelerated spreading (infinite spreading speed) can occur if $J_1$ and $J_2$ violate the following condition:
\begin{enumerate}
\item[{\bf(J1)}] $\int_0^{\yy}xJ_i(x)\dx<\yy$\, for $i=1,2$.
 \end{enumerate}
However, the rate of accelerated spreading, a hot topic in spreading phenomenon modelled by nonlocal diffusion equation, is not discussed here and left to future work. We note that for \eqref{1.7}, by virtue of some subtle upper and lower solutions, Du and Ni \cite{DN3,DN4,DN5} obtained some sharp estimates on the rate of accelerated spreading for a class of algebraic decay kernels.

Before stating the conclusion of this section, we first consider the following semi-wave problem
 \bes\left\{\begin{aligned}\label{4.1}
&d_1\int_{-\yy}^0\!J_1(x-y)p(y)\dy-d_1p+cp'-ap+H(q)=0, & & x\in(-\yy,0),\\
&d_2\int_{-\yy}^0\!J_2(x-y)q(y)\dy-d_2q+cq'-bq+G(p)=0, & & x\in(-\yy,0),\\
&p(-\yy)=U, ~ ~ q(-\yy)=V, ~ ~p(0)=q(0)=0,\\
&c=\int_{-\yy}^0\int_0^{\yy}\!\big[\mu_1J_1(x-y)p(x)
+\mu_2J_2(x-y)q(x)\big]\dy\dx.
 \end{aligned}\right.
 \ees

\begin{proposition}{\rm (\cite[Theorem 1.2]{DN2})} \label{p4.1}Problem \eqref{4.1} has a unique solution triplet $(\tilde c,\tilde p,\tilde q)$ with $\tilde c>0$ and $\tilde p, \tilde q$ strictly decreasing in $(-\yy,0]$ if and only if {\bf (J1)} holds.
\end{proposition}

The following is our main conclusion of this section.

\begin{theorem}[Spreading speed]\label{t1.4} Let $(u,v,h)$ be the unique solution of \eqref{1.10} and spreading happen. Then the following statements are valid.
 \begin{enumerate}[$(1)$]
 \item If $({\bf J1})$ is satisfied, then
 \[\dd\lim_{t\to\yy}\frac{h(t)}{t}=\tilde c, ~ \dd\lim_{t\to\yy}\max_{[0,ct]}\big(|u(t,x)-U|+|v(t,x)-V|\big)=0 \; ~ {\rm for ~ any ~ } c\in[0,\tilde c),\]
 and for any $\tau \in(0,1)$,
  \bess
 \lim_{t\to\yy}\frac{\min\{x>0:u(t,x)=\tau U\}}{t}=\lim_{t\to\yy}\frac{\min\{x>0:v(t,x)=\tau V\}}{t}=\tilde c,
  \eess
 where $(U,V)$ is determined by \eqref{1.2} and $\tilde c$ is uniquely given by semi-wave problem \eqref{4.1}.
 \item If $({\bf J1})$ is violated, then
  \[\lim_{t\to\yy}\frac{h(t)}{t}=\yy, ~ \dd\lim_{t\to\yy}\max_{[0,ct]}\big(|u(t,x)-U+|v(t,x)-V|\big)=0\; ~ {\rm for ~ any ~ } c\in[0,\yy),\]
 and for any $\tau\in(0,1)$,
  \bess
 \lim_{t\to\yy}\frac{\min\{x>0:u(t,x)=\tau U\}}{t}=\yy,\;\;\;
   \lim_{t\to\yy}\frac{\min\{x>0:v(t,x)=\tau V\}}{t}=\yy.
 \eess
 \end{enumerate}
\end{theorem}

We shall prove Theorem \ref{t1.4} by using solutions of problem \eqref{4.1} and its variations to build suitable upper and lower solutions. The proof is divided into several lemmas.

\begin{lemma}\label{l4.1}Suppose that {\bf (J1)} holds. Let $(u,v,h)$ be the unique solution of \eqref{1.10}. Then $\limsup_{t\to\yy}\frac{h(t)}{t}\le \tilde c$,
where $\tilde c$ is uniquely given by Proposition \ref{p4.1}.
\end{lemma}

\begin{proof} Define $\bar{h}(t)=(1+\ep)\tilde ct+L$, $\bar{u}(t,x)=(1+\ep)\tilde p(x-\bar{h}(t))$, $\bar{v}(t,x)=(1+\ep)\tilde q(x-\bar{h}(t))$,
where $0<\ep\ll 1$ and $L>0$ is a positive constant to be determined later. We now prove that there exist suitable $L$ and $T$ such that $(\bar{u},\bar{v},\bar{h})$ satisfies
 \bes\left\{\begin{aligned}\label{4.2}
&\bar{u}_t\ge d_1\int_0^{\bar{h}(t)}\!\!J_1(x-y)\bar{u}(t,y)\dy-d_1j_1(x)\bar{u}-a\bar{u}+H(\bar{v}), \hspace{2mm} t>0, ~ x\in[0,\bar{h}(t)),\\
 &\bar{v}_t\ge d_2\int_0^{\bar{h}(t)}\!\!J_2(x-y)\bar{v}(t,y)\dy-d_2j_2(x)\bar{v}-b\bar{v}+G(\bar{u}), \hspace{3mm} t>0, ~ x\in[0,\bar{h}(t)),\\
&\bar{u}(t,\bar{h}(t))\ge0, ~ \bar{v}(t,\bar{h}(t))\ge0, \hspace{48mm} t>0,\\
&\bar h'(t)\ge \int_0^{\bar{h}(t)}\!\!\int_{\bar{h}(t)}^{\yy}\big[\mu_1J_1(x-y)\bar{u}(t,x)
+\mu_2J_2(x-y)\bar{v}(t,x)\big]\dy\dx,\;\;t>0,\\
&\bar{h}(0)\ge h(T),~ \bar{u}(0,x)\ge u(T,x),\;\; \bar{v}(0,x)\ge v(T,x), ~\;\; x\in[0,h(T)].
 \end{aligned}\right.
 \ees
Once this is done, by comparison principle, we derive that $\bar{h}(t)\ge h(t+T)$, $\bar{u}(t,x)\ge u(t+T,x)$ and $\bar{v}(t,x)\ge v(t+T,x)$ for $t\ge0$ and $x\in[0,h(t+T)]$, which indicates $\limsup_{t\to\yy}\frac{h(t)}{t}\le(1+\ep)\tilde c$. By the arbitrariness of $\ep$, the desired result holds. Thus, it suffices to verify \eqref{4.2}.

Let us begin with proving the first two inequalities in \eqref{4.2}. Notice that  $\tilde p(x), \tilde q(x)$ are strictly decreasing in $x<0$, and ${H(z)}/{z}$ is decreasing and ${G(z)}/{z}$ is strictly decreasing in $z>0$. To save space, in this part we set $1+\ep=\gamma$ and $\rho=\rho(x,t)=x-\bar{h}(t)$. Direct computations yield that, for $t>0$ and $x\in[0,\bar{h}(t))$,
 \bes
&&\frac 1\gamma\kk(\bar{u}_t-d_1\int_0^{\bar{h}(t)}\!J_1(x-y)\bar{u}(t,y)\dy
+d_1j_1(x)\bar{u}+a\bar{u}-H(\bar{v})\rr)\nonumber\\[1mm]
&=&-\gamma \tilde c\tilde p'(\rho)-d_1\int_0^{\bar{h}(t)}
\!J_1(x-y)\tilde p(y-\bar{h}(t))\dy+d_1j_1(x)\tilde p(\rho)+a\tilde p(\rho)-\frac 1\gamma
H(\gamma\tilde q(\rho))\nonumber\\[1mm]
&\ge&-\tilde c\tilde p'(\rho)-d_1\int_0^{\bar{h}(t)}
\!J_1(x-y)\tilde p(y-\bar{h}(t))\dy
+d_1j_1(x)\tilde p(\rho)+a\tilde p(\rho)-\frac 1\gamma H(\gamma\tilde q(\rho))\nonumber\\[1mm]
&=&d_1\int_{-\yy}^0\!J_1(x-\bar{h}(t)-y)\tilde p(y)\dy
-d_1\tilde p(\rho)-a\tilde p(\rho)+H(\tilde q(\rho))\nonumber\\[1mm]
&&-d_1\int_0^{\bar{h}(t)}\!J_1(x-y)\tilde p(y-\bar{h}(t))\dy
+d_1j_1(x)\tilde p(\rho)+a\tilde p(\rho)-\frac 1 \gamma H(\gamma\tilde q(\rho))\nonumber\\[1mm]
&=&d_1\kk(\int_{-\yy}^0\!J_1(x-y)\big[\tilde p(y-\bar{h}(t))
-\tilde p(\rho)\big]\rr)\dy+H(\tilde q(\rho))-\frac 1\gamma H(\gamma\tilde q(\rho))\ge 0.\lbl{4.3a}
 \ees
Similarly, we can prove the second inequality of \eqref{4.2}. From our definitions of $\bar{u}$ and $\bar{v}$, it is clear that $\bar{u}(t,\bar{h}(t))=\bar{v}(t,\bar{h}(t))=0$ for $t>0$. Then we check the fourth inequality in \eqref{4.2}. Simple calculations show
 \bess
&&\int_0^{\bar{h}(t)}\!\int_{\bar{h}(t)}^{\yy}\big[\mu_1J_1(x-y)\bar{u}(t,x)
+\mu_2J_2(x-y)\bar{v}(t,x)\big]\dy\dx\\
&\le&(1+\ep)\int_{-\yy}^0\int_0^{\yy}\big[\mu_1J_1(x-y)\tilde p(x)+\mu_2J_2(x-y)\tilde q(x)\big]\dy\dx=(1+\ep)\tilde c=\bar{h}'(t).
 \eess

It remains to show the inequalities in the last two lines of \eqref{4.2}. Let $(\ol u(t), \ol v(t))$ be the unique solution of the corresponding ODE of \eqref{1.10} with $(\ol u(0),\ol v(0))=(\|u_0\|_\yy,\|v_0\|_\yy)$. Under the condition {\bf(H1)}, we can show, by phase plane
analysis, that $\lim_{t\to\yy}(\ol u(t), \ol v(t))=(0,0)$ if $\mathcal{R}_0<1$, and $\lim_{t\to\yy}(\ol u(t), \ol v(t))=(U,V)$ if $\mathcal{R}_0>1$. Moreover,
by a simple comparison argument, we know that the solution component $(u,v)$ satisfy that $(u, v)\le(\ol u, \ol v)$. Consequently, $\limsup_{t\to\yy}u\le U$ and $\limsup_{t\to\yy}v\le V$ uniformly in $x\in[0,\yy)$. For the given $\ep>0$, we can find $T>0$ such that $u\le (1+\ep/2)U$ and $v\le (1+\ep/2)V$ for $t\ge T$ and $x\ge0$. There exists $L\gg h(T)$, such that $\bar{u}(0,x)=(1+\ep)\tilde p(x-L)\ge(1+\ep/2)U\ge u(t+T,x)$ and $\bar{v}(0,x)=(1+\ep)\tilde q(x-L)\ge(1+\ep/2)V\ge v(t+T,x)$ for $t\ge0$ and $x\in[0,h(T)]$. Inequalities in the last two lines of \eqref{4.2} are verified. Therefore, \eqref{4.2} holds and the proof is complete.
\end{proof}

Then we prove the lower limit of $h(t)$ which will be handled by several lemmas. Due to $\mathcal{R}_0>1$, there exists a $\sigma_0>0$ such that $\frac{H'(0)G'(0)}{(a+\sigma)(b+\sigma)}>1$ for all $\sigma\in(0,\sigma_0)$. Then obviously, the system
  \[(a+\sigma)u=H(v), ~ ~ (b+\sigma)v=G(u)\]
has a unique positive root $(U_\sigma,V_\sigma)$ with $U>U_\sigma$ and $V>V_\sigma$. By Proposition \ref{p4.1}, the corresponding semi-wave problem
\bes\left\{\begin{aligned}\label{4.3}
&d_1\int_{-\yy}^0J_1(x-y)p(y)\dy-d_1p+cp'-(a+\sigma)p+H(q)=0, & & x\in(-\yy,0),\\
&d_2\int_{-\yy}^0J_2(x-y)q(y)\dy-d_2q+cq'-(b+\sigma)q+G(p)=0, & & x\in(-\yy,0),\\
&p(-\yy)=U_\sigma, ~ ~ q(-\yy)=V_\sigma, ~ ~p(0)=q(0)=0,\\
&c=\int_{-\yy}^0\int_0^{\yy}\big[\mu_1J_1(x-y)p(x)
+\mu_2J_2(x-y)q(x)\big]\dy\dx
 \end{aligned}\right.
 \ees
has a unique solution triplet $(\tilde c_{\sigma},\tilde p_\sigma,\tilde q_\sigma)$, where $\tilde c_{\sigma}>0$, and both $\tilde p_\sigma$ and $\tilde q_\sigma$ are strictly decreasing in $(-\yy,0]$ if and only if {\bf(J1)} holds.

\begin{lemma}\label{l4.2}Assume that {\bf (J1)} holds. Then $\tilde c_{\sigma}\to \tilde c$, $(\tilde p_\sigma, \tilde q_\sigma)\to (\tilde p, \tilde q)$ in $[C_{\rm loc}([0,\yy))]^2$ as $\sigma\to0$.
 \end{lemma}

\begin{proof}Let $\{\sigma_n\}\subseteq(0,\sigma_0)$ with $\sigma_n$ decreasing to $0$, and denote $(\tilde c_{\sigma_n},\tilde p_{\sigma_n},\tilde q_{\sigma_n})$ by $(\tilde c_n,\tilde p_n,\tilde q_n)$. Similarly to \cite[Lemma 2.8]{DN2}, we have $(\tilde c_n, \tilde p_n, \tilde q_n)\le (\tilde c_{n+1}, \tilde p_{n+1}, \tilde q_{n+1})\le (\tilde c, \tilde p,\tilde q)$. Thus we can define $(\bar c, \bar p, \bar q) =\dd\lim_{n\to\yy}(\tilde c_n, \tilde p_n, \tilde q_n)$ with $\bar c\in(0,\tilde c]$. Obviously, $\bar p(x)$ and $\bar q(x)$ are decreasing in $(-\yy,0]$. For any $x<0$, integrating the first equality of \eqref{4.3} leads to
 \bess
 \tilde c_n\tilde p_n(0)-\tilde c_n\tilde p_n(x)
 =\int_0^{x}\!\kk(\!d_1\int_{-\yy}^0\!J_1(z-y)\tilde p_n(y)\dy
 -d_1\tilde p_n(z)-(a+\sigma_n)\tilde p_n(z)+H(\tilde q_n(z))\!\rr){\rm d}z.
 \eess
Letting $n\to\yy$ and using the dominated convergence theorem, we have
 \[\bar c\bar p(0)-\bar c\bar p(x)=\dd\int_0^{x}
 \kk(d_1\int_{-\yy}^0\!J_1(z-y)\bar p(y)\dy-d_1\bar p(z)
 -a\bar p(z)+H(\bar q(z))\rr){\rm d}z.\]
Differentiating the above equality yields
 \[-\bar c\bar p'(x)=d_1\int_{-\yy}^0J_1(x-y)\bar p(y)\dy-d_1\bar p(x)
 -a\bar p(x)+H(\bar q(x)).\]
Similarly, we have
  \[d_2\int_{-\yy}^0J_2(x-y)\bar q(y)\dy-d_2\bar q(x)+\bar c\bar q'(x)-b\bar q(x)+G(\bar p(x))=0, ~ x<0.\]
Notice that $\tilde p_n\le\bar p\le\tilde p$, $\tilde q_n\le\bar  q\le\tilde q$, and $\tilde p_n(-\yy)=K^{\sigma_n}_1\to U=\tilde p(-\yy)$, $\tilde q_n(-\yy)=K^{\sigma_n}_2\to V=\tilde q(-\yy)$ as $n\to\yy$. We easily derive that $\bar p(-\yy)=U$, $\bar q(-\yy)=V$.

Moreover, by monotone convergence theorem, we have that as $n\to\yy$,
 \bess
 \tilde c_n&=&\int_{-\yy}^0\int_0^{\yy}\big[\mu_1J_1(x-y)\tilde
 p_n(x)+\mu_2J_2(x-y)\tilde q_n(x)\big]\dy\dx\\
 &\to&\int_{-\yy}^0\int_0^{\yy}\big[\mu_1J_1(x-y)\bar p(x)+\mu_2J_2(x-y)\bar q(x)\big]\dy\dx=\bar c.
 \eess
Taking advantage of Proposition \ref{p4.1}, we have $\bar c=\tilde c$, and $\bar p(x)=\tilde p(x)$, $\bar q(x)=\tilde q(x)$. Together with Dini's theorem, we have  $\tilde p_n\to\tilde p$ and $\tilde q_n\to\tilde q$ in $C_{\rm loc}([0,\yy))$ which completes the proof.
 \end{proof}

For $n\ge1$, define
 \bess
 &\xi(x)=1, ~ |x|\le1; ~ ~ ~ \xi(x)=2-|x|, ~ 1<|x|\le2; ~ ~ ~  \xi(x)=0, ~ |x|>2,\\
 &J^n_i(x)=J_i(x)\xi(\frac{x}{n}), ~ j^n_i(x)=\dd\int_0^{\yy}J^n_i(x-y)\dy.
 \eess
Then it is not hard to verify that $J^n_i$ is supported compactly, increasing in $n$ and $J^n_i\le J_i$ for $x\in\mathbb{R}$. What's more, $J^n_i\to J_i$ in $L^1(\mathbb{R})$ and $C_{\rm loc}(\mathbb{R})$, and $j^n_i\to j_i$ in $L^{\yy}(\mathbb{R})$ as $n\to\yy$. For any $\sigma\in(0,\sigma_0)$, we can choose $n$ large enough, say $n\ge N$, such that $d_i(j^n_i(x)-j_i(x))+\sigma\ge0$ in $\mathbb{R}$ and \[\frac{H'(0)G'(0)}{\big[a+\sigma+d_1(1-\|J^n_1\|_1)\big]
\big[b+\sigma+d_2(1-\|J^n_2\|_1)\big]}>1.\]
  Consider the following semi-wave problem
 \bes\left\{\begin{aligned}\label{4.4}
&d_1\int_{-\yy}^0J^n_1(x-y)p(y)\dy-d_1p+cp'-(a+\sigma)p+H(q)=0, & & x\in(-\yy,0),\\
&d_2\int_{-\yy}^0J^n_2(x-y)q(y)\dy-d_2q+cq'-(b+\sigma)q+G(p)=0, & & x\in(-\yy,0),\\
&p(-\yy)=U_\sigma^n, ~ ~ q(-\yy)=V_\sigma^n, ~ ~p(0)=q(0)=0,\\
&c=\int_{-\yy}^0\int_0^{\yy}\!\big[\mu_1J^n_1(x-y) p(x)+
\mu_2J^n_2(x-y) q(x)\big]\dy\dx,
 \end{aligned}\right.
 \ees
where $\sigma\in(0,\sigma_0)$, and $(U_\sigma^n,V_\sigma^n)$ is the unique positive root of
\[d_1(\|J^n_1\|_1-1)u-(a+\sigma)u+H(v)=0, ~ ~ d_2(\|J^n_2\|_1-1)v-(b+\sigma)v+G(u)=0.\]
Note that both $J^n_1$ and $J^n_2$ are supported compactly. In view of Proposition \ref{l4.1}, problem \eqref{4.4} has a unique solution triplet $(\tilde c^n_\sigma,\,\tilde p^n_\sigma,\,\tilde q^n_\sigma)$.

\begin{lemma}\label{l4.3}If {\bf (J1)} holds, then $\tilde c^n_\sigma\to \tilde c_{\sigma}$, $\tilde p^n_\sigma\to\tilde p_{\sigma}$ and $\tilde q^n_\sigma\to\tilde q_{\sigma}$ in $C_{\rm loc}((-\yy,0])$ as $n\to\yy$. Moreover, if {\bf (J1)} does not hold, then $\tilde c^n_\sigma\to\yy$ as $n\to\yy$.
\end{lemma}

\begin{proof}
Recall that $J^n_i$ is increasing in $n\ge1$, $J^n_i\le J_i$ for $x\in\mathbb{R}$, and $J^n_i\to J_i$ in $L^1(\mathbb{R})$ and $C_{\rm loc}(\mathbb{R})$. Then following the similar method as in the proof of Lemma \ref{l4.2}, we can prove the first assertion and thus the details are ignored here.

We now show the second assertion. Notice that {\bf (J1)} is violated. Without loss of generality, we assume that $\int_0^{\yy}xJ_1(x)\dx=\yy$. Obviously, $\tilde p^n_\sigma$ is increasing in $n$ and $0\le\tilde p^n_\sigma\le U_\sigma$ in $(-\yy,0]$. Thus, we can define $\bar p_{\sigma}=\dd\lim_{n\to\yy}\tilde p^n_\sigma$. Using $\tilde p^n_\sigma\ge\tilde p^1_\sigma$ for $n\ge 1$, we have that, for any $l>l_0>0$,
\bess
\liminf_{n\to\yy}\int_{-\yy}^0\int_0^{\yy}\!\!J^n_1(x-y)\tilde p^n_\sigma(x)\dy\dx&\ge&\liminf_{n\to\yy}\int_{-l}^{-l_0}\int_0^lJ^n_1(x-y)\tilde p^1_\sigma(x)\dy\dx\\
&=&\int_{-l}^{-l_0}\!\int_0^lJ_1(x-y)\tilde p^1_\sigma(x)\dy\dx\\
&\ge&\tilde p^1_\sigma(-l_0)\int_{-l}^{-l_0}\!\int_{-x}^{l-x}J_1(y)\dy\dx\\
&\ge&\tilde p^1_\sigma(-l_0)\int_{l_0}^l\int_{-y}^{-l_0}J_1(y)\dx\dy\\
&=&\tilde p^1_\sigma(-l_0)\int_{l_0}^lJ_1(y)(y-l_0)\dy\\
 &\to&\yy ~ ~ ~ {\rm as } ~ l\to\yy,
\eess
 which, combined with
\bess
\liminf_{n\to\yy}\tilde c^n_\sigma\ge\liminf_{n\to\yy}\mu_1\int_{-\yy}^0\int_0^{\yy}J^n_1(x-y)\tilde p^n_\sigma(x)\dy\dx,
\eess
yields $\tilde c^n_\sigma\to\yy$ as $n\to\yy$. The proof is complete.
\end{proof}

For $n\ge N$, we consider the following auxiliary problem
 \bes\left\{\!\begin{aligned}\label{4.5}
&(u^n_\sigma)_t=d_1\int_0^{h^n_\sigma(t)}\!\!J^n_1(x\!-\!y)u^n_\sigma(t,y)\dy
\!-\!d_1j^n_1u^n_\sigma\!-\!(a\!+\!\sigma)u^n_\sigma\!+\!H(v^n_\sigma), \;\; t>0, ~ x\in[0,h^n_\sigma(t)),\\
&(v^n_\sigma)_t=d_2\int_0^{h^n_\sigma(t)}\!\!J^n_2(x\!-\!y)v^n_\sigma(t,y)\dy
\!-\!d_2j^n_2v^n_\sigma\!-\!(b\!+\!\sigma)v^n_\sigma\!+\!G(u^n_\sigma), \;\; t>0, ~ x\in[0,h^n_\sigma(t)),\\
&u^n_\sigma(t,h^n_\sigma(t))=0, ~ v^n_\sigma(t,h^n_\sigma(t))=0, \hspace{2mm} t>0,\\
&(h^n_\sigma)'(t)= \int_0^{h^n_\sigma(t)}\!\!\int_{h^n_\sigma(t)}^{\yy}\big[\mu_1J^n_1(x-y)u^n_\sigma(t,x)
+\mu_2J^n_2(x-y)v^n_\sigma(t,x)\big]\dy\dx, \;\; t>0,\\
&h^n_\sigma(0)=h(T), ~ u^n_\sigma(0,x)=u(T,x), ~ v^n_\sigma(0,x)=v(T,x), ~ x\in[0,h(T)].
 \end{aligned}\right.
 \ees
Using the same arguments as in the proofs of Lemmas \ref{l3.1} and \ref{l3.5}, we can show that there exists a critical value $\ell^{n,\sigma}_*>0$, depending only on $J^n_i$, $H'(0)$, $G'(0)$ and parameters in the first two equalities in \eqref{4.5}, such that spreading happens if $T$ large enough satisfying $h^n_\sigma(0)=h(T)\ge\ell^{n,\sigma}_*$.

 \begin{lemma}\label{l4.4} Let $(u^n_\sigma,v^n_\sigma,h^n_\sigma)$ be the unique solution of \eqref{4.5}. Then we have
 \bess
 \dd\liminf_{t\to\yy}\frac{h^n_\sigma(t)}{t}\ge \tilde c^n_\sigma,\;\;\;
 \dd\liminf_{t\to\yy}(u^n_\sigma,v^n_\sigma)\ge(U_\sigma^n,V_\sigma^n) ~ ~ {\rm uniformly ~ in ~ }x\in[0,ct], ~ \forall c\in[0,\tilde c^n_\sigma).
  \eess
 \end{lemma}

\begin{proof} Define $\underline{h}(t)=(1-\ep)\tilde c^n_\sigma t+2L$, $\underline{u}(t,x)=(1-\ep)\tilde p^n_\sigma(x-\underline{h}(t))$ and $ \underline{v}(t,x)=(1-\ep)\tilde q^n_\sigma(x-\underline{h}(t))$,
where $L>0$ is large enough such that $J^n_i(x)=0$ for $|x|\ge L$ and $\ep>0$ is arbitrarily small. We next show that there exists a $T_1>0$ such that $(\ud{u},\ud{v},\ud{h})$ satisfies
  \bes\left\{\!\begin{aligned}\label{4.6}
 &\ud u_t\le d_1\int_0^{\ud h(t)}\!\!\!J^n_1(x\!-\!y)\ud u(t,y)\dy\!-\!d_1j^n_1\ud u\!-\!(a+\sigma)\ud u+H(\ud v), \hspace{2mm} t>0, ~ x\in[L,\ud h(t)),\\
&\ud v_t\le d_2\int_0^{\ud h(t)}\!\!\!J^n_2(x\!-\!y)\ud v(t,y)\dy\!-\!d_2j^n_2\ud v\!-\!(b+\sigma)\ud v+G(\ud u),\hspace{2mm} t>0, ~ x\in[L,\ud h(t)),\\
&\ud u(t,\ud h(t))=0, ~ \ud v(t,\ud h(t))=0, \hspace{4mm} t>0,\\
&\ud h'(t)\le\!\int_0^{\ud h(t)}\!\!\!\int_{\ud h(t)}^{\yy}\!\big[\mu_1J^n_1(x\!-\!y)\ud u(t,x)
\!+\!\mu_2J^n_2(x\!-\!y)\ud v(t,x)\big]\dy\dx, \;\; t>0,\\
&\ud u(t,x)\le u^n_\sigma(t+T_1,x), ~ \ud v(t,x)\le v^n_\sigma(t+T_1,x), \hspace{2mm} t>0,~ x\in[0,L],\\
&\ud h(0)\le h^n_\sigma(T_1), ~ \ud u(0,x)\le u^n_\sigma(T_1,x),
 ~ \ud v(0,x)\le v^n_\sigma(T_1,x), \hspace{2mm} x\in[0,h^n_\sigma(T_1)].
 \end{aligned}\right.
 \ees
Once it is done, by a comparison argument, we have $h^n_\sigma(t+T_1)\ge\ud h(t)$ for $t\ge0$. The arbitrariness of $\ep$ implies the first assertion.

 For second assertion, we can choose $\ep$ sufficiently small such that $(1-\ep)\tilde c^n_\sigma>c$. Due to the definitions of $\ud{u}$ and $\ud{v}$, it is easy to see that $\ud{u}\to (1-\ep)U_\sigma^n$ and $\ud{v}\to (1-\ep)V_\sigma^n$ uniformly in $x\in[0,ct]$ as $t\to\yy$. So for any small $\ep>0$, we have $
 \liminf_{t\to\yy}u^n_\sigma\ge(1-\ep)U_\sigma^n$ and $\liminf_{t\to\yy}v^n_\sigma\ge(1-\ep)V_\sigma^n$ uniformly in $x\in[0,ct]$,
 which together with the arbitrariness of $\ep$ yields the second assertion.

It remains to prove \eqref{4.6}. Since spreading happens for $(u^n_\sigma,v^n_\sigma,h^n_\sigma)$, we can choose a large $T_1$ such that $h^n_\sigma(T_1)>2L=\ud{h}(0)$,
  $\ud{u}(t,x)\le(1-\ep)U_\sigma^n\le u^n_\sigma(t+T_1,x)$ and $\ud{v}(t,x)\le(1-\ep)V_\sigma^n\le v^n_\sigma(t+T_1,x)$
for $t>0$ and $x\in[0,2L]$. Noticing that $\ud u(0,x)=(1-\ep)\tilde p^n_\sigma(x-2L)=0$ and $\ud v(0,x)=(1-\ep)\tilde q^n_\sigma(x-2L)=0$ for $x\ge 2L$. Inequalities in the last two lines of \qq{4.6} hold true.

Recall $J^n_i(x)=0$ for $|x|\ge L$. Simple computations yield
 \bess
 &&\int_0^{\ud h(t)}\!\!\!\int_{\ud h(t)}^{\yy}\big[\mu_1J^n_1(x-y)\ud u(t,x)+\mu_2 J^n_2(x-y)\ud v(t,x)\big]\dy\dx\\
 &=&(1-\ep)\int_{-\yy}^0\int_0^{\yy}\big[\mu_1J^n_1(x-y)
 \tilde p^n_\sigma(x)+\mu_2J^n_2(x-y)\tilde q^n_\sigma(x)\big]\dy\dx\\
 &=&(1-\ep)\tilde c^n_\sigma=(\ud{h})'(t).
 \eess
The inequality in forth line of \qq{4.6} is verified.

For $t>0$ and $L\le x<\ud{h}(t)$, since $j^n_1(x)\le 1$ and $H(z)/z$ is decreasing in $z>0$, similar to the derivation of \qq{4.3a}, it can be obtained that
 \bess
 \ud u_t
 &\le&d_1\int_{-\yy}^{\ud{h}(t)}\!\!J^n_1(x-y)\ud{u}(t,y)\dy
 -d_1\ud{u}-(a+\sigma)\ud{u}+(1-\ep)H(\tilde q^n_\sigma(x-\ud{h}))\\
 &\le& d_1\int_0^{\ud{h}(t)}\!\!J^n_1(x-y)\ud{u}(t,y)\dy
 -d_1j^n_1(x)\ud{u}-(a+\sigma)\ud{u}+H(\ud{v}).
 \eess
 The first inequality of \eqref{4.6} is obtained. Analogously, we can argue as above to deduce the second inequality in \eqref{4.6}. Therefore, \eqref{4.6} holds and the proof is finished.
 \end{proof}

\begin{lemma}\label{l4.5} The unique solution $(u^n_\sigma,v^n_\sigma,h^n_\sigma)$ of \eqref{4.5} is a lower solutions of \eqref{1.10}.
\end{lemma}

\begin{proof}
Recall that $d_i(j^n_i(x)-j_i(x))+\sigma\ge0$ in $\mathbb{R}$ and $J^n_i\le J_i$. We can see that, for $t>0$ and $x\in[0,h^n_\sigma(t))$,
  \bess
(u^n_\sigma)_t&=&d_1\int_0^{h^n_\sigma(t)}\!\!J^n_1(x-y)u^n_\sigma(t,y)\dy
-d_1j^n_1(x)u^n_\sigma-(a+\sigma)u^n_\sigma+H(v^n_\sigma)\\
&\le& d_1\int_0^{h^n_\sigma(t)}\!\!J_1(x\!-\!y)u^n_\sigma(t,y)\dy\!-\!
d_1j_1(x)u^n_\sigma\!+\!(d_1j_1(x)\!-\!d_1j^n_1(x)\!-\!\sigma)u^n_\sigma\!-\!
au^n_\sigma\!+\!H(v^n_\sigma)\\
&\le& d_1\int_0^{h^n_\sigma(t)}\!\!J_1(x-y)u^n_\sigma(t,y)\dy-d_1j_1(x)u^n_\sigma-au^n_\sigma+H(v^n_\sigma).
\eess
Similarly, we have
\[(v^n_\sigma)_t\le d_2\int_0^{h^n_\sigma(t)}J_2(x-y)v^n_\sigma(t,y)\dy-d_2j_2(x)v^n_\sigma-bv^n_\sigma+G(u^n_\sigma).\]
Moreover,
  \bess
(h^n_\sigma)'(t)&=&\int_0^{h^n_\sigma(t)}\!\!\int_{h^n_\sigma(t)}^{\yy}
\big[\mu_1J^n_1(x-y)u^n_\sigma(t,x)+\mu_2J^n_2(x-y)v^n_\sigma(t,x)\big]\dy\dx\\
&\le&\int_0^{h^n_\sigma(t)}\!\!\int_{h^n_\sigma(t)}^{\yy}
\big[\mu_1J_1(x-y)u^n_\sigma(t,x)+\mu_2J_2(x-y)v^n_\sigma(t,x)\big]\dy\dx.
\eess
By a comparison method, we completes the proof.
\end{proof}

\begin{lemma}\label{l4.6}Let $(u,v,h)$ be the unique solution of \eqref{1.10}. Then the following statements are valid.\vspace{-1.5mm}
\begin{enumerate}[$(1)$]
  \item If {\bf (J1)} holds, $\liminf\limits_{t\to\yy}\frac{h(t)}{t}\ge \tilde c$ and $\dd\liminf_{t\to\yy}(u,v)\ge(U,V)$ uniformly in $x\in[0,ct]$ for $c\in[0,\tilde c)$.\vspace{-1.5mm}
  \item If {\bf (J1)} is violated, $\lim\limits_{t\to\yy}\frac{h(t)}{t}=\yy$ and $\dd\liminf_{t\to\yy}(u,v)\ge(U,V)$ uniformly in $x\in[0,ct]$ for $c\ge0$.\vspace{-1.5mm}
\end{enumerate}
\end{lemma}
\begin{proof}(1) By Lemmas \ref{l4.4} and \ref{l4.5}, we have $\liminf\limits_{t\to\yy}\frac{h(t)}t\ge \tilde c^n_\sigma$. Together with Lemmas \ref{l4.2} and \ref{l4.3}, we further derive $\liminf\limits_{t\to\yy}\frac{h(t)}t\ge \tilde c$. Again from Lemma \ref{l4.4} and \ref{l4.5}, we have $\liminf_{t\to\yy}(u,v)\ge(U_\sigma^n,V_\sigma^n)$ uniformly in $x\in[0,ct]$ for all  $c\in[0,\tilde c^n_\sigma)$,
which combined with the fact that $(U_\sigma^n,V_\sigma^n)$ is increasing to $(U_\sigma,V_\sigma)$ as $n\to\yy$, and $(U_\sigma,V_\sigma)$ is decreasing to $(U,V)$ as $\sigma\to0$, yields that $\liminf_{t\to\yy}(u,v)\ge(U,V)$ uniformly in $x\in[0,ct]$ for all $c\in[0,\tilde c)$.

(2) Notice that {\bf (J1)} does not hold. By Lemma \ref{l4.3}, $\tilde c^n_\sigma\to\yy$ as $n\to\yy$. Thus this assertion directly follows from the similar analysis as above. We complete the proof.
\end{proof}

Theorem \ref{t1.4} follows from Lemmas \ref{l4.1} and \ref{l4.6}, as well as the result (already proved in the proof of Lemma \ref{l4.1}) $\limsup_{t\to\yy}u\le U$ and $\limsup_{t\to\yy}v\le V$ uniformly in $x\in[0,\yy)$.

\section{Appendix A}
{\setlength\arraycolsep{2pt}
\setcounter{equation}{0}

From the point of view of PDEs, the differential and boundary operators in problems \qq{1.7} and/or \qq{1.8} are determined by diffusion coefficient $d$, kernel function $J$, and the moving coefficient $\mu$ of the free boundary. The triplet $(d,J,\mu)$ can be seen as the ``working operator" for solving problems \qq{1.7} and/or \qq{1.8}. In this appendix we shall show that \qq{1.8} cannot be transformed into \qq{1.7} in the sense of ``working operator".

Let $(u,h)$ be the unique solution of \eqref{1.8}. Define $\tilde{u}_0(x)=u_0(|x|)$ and $\tilde{u}(t,x)=u(t,|x|)$. We next prove that there is no $\tilde d>0$, $\tilde\mu>0$ and $\tilde J(x)$ satisfying {\bf(J)} such that $(\tilde{u},-h,h)$ is the unique solution of \eqref{1.7} with $d$, $\mu$, $J$ and $u_0(x)$ replaced by $\tilde d$, $\tilde\mu$, $\tilde J(x)$ and $\tilde{u}_0(|x|)$, respectively.

\begin{theorem}\lbl{ta.1} Fix $d, \mu>0$ and $J(x)$. There do not exist $\tilde d, \tilde\mu>0$ and $\tilde J(x)$ satisfying condition {\bf(J)} such that, for all $u_0(x)$ satisfying {\bf (I)} and $h_0>0$, $(\tilde u, -h, h)$ is the unique solution of \qq{1.7} with $(d, \mu, J)$ replaced by $(\tilde d, \tilde\mu, \tilde J)$, respectively.
 \end{theorem}

\begin{proof} Assume on the contrary that there exists such triplet $(\tilde d, \tilde\mu, \tilde J(x))$ as desired. Simple computations yield that for $t>0$ and $x\in[0,h(t))$,
 \bess
\tilde{u}_t&=&\tilde d\int_{-h(t)}^{h(t)}\tilde J(x-y)\tilde{u}(t,y)\dy-\tilde d\tilde{u}(t,x)+f(\tilde{u})\\
&=&\tilde d\int_0^{h(t)}\big[\tilde J(x-y)+\tilde J(x+y)\big]u(t,y)\dy-\tilde du+f(u).
 \eess
Since $\tilde{u}(t,x)=u(t,x)$ for $t>0$ and $x\in[0,h(t))$, we have $\tilde{u}_t=u_t$ in such regions. Thus, by the differential equation of $u$,
\[\tilde d\int_0^{h(t)}\big[\tilde J(x-y)+\tilde J(x+y)\big]u(t,y)\dy-\tilde du=d\int_0^{h(t)}J(x-y)u(t,y)-dj(x)u.\]
By continuity and $u(t,h(t))=0$ for $t\ge0$,
\[\tilde d\int_0^{h(t)}\big[\tilde J(h(t)-y)+\tilde J(h(t)+y)\big]u(t,y)\dy=d\int_0^{h(t)}J(h(t)-y)u(t,y)\dy.\]
Letting $t\to0$ and using continuity again, we obtain
\bes\label{5.1}
\tilde d\int_0^{h_0}\big[\tilde J(h_0-y)+\tilde J(h_0+y)\big]u_0(y)\dy=d\int_0^{h_0}J(h_0-y)u_0(y)\dy
\ees
holds for all $h_0>0$ and $u_0(x)$ satisfying {\bf (I)}. For $h_0>1$, choose a class of $u_0(x)$ as follows
\[u_0(x)=1, ~ ~ 0\le x\le h_0-1/{h_0}; ~ ~ ~ u_0(x)=h_0(h_0-x), ~ ~ h_0-1/{h_0}\le x\le h_0.\]
Substituting such $u_0$ into \eqref{5.1} and then direct calculating yield
 \bess
&&\tilde d\int_{-h_0}^{-1/{h_0}}\!\!\tilde J(y)\dy+\tilde d\int_{h_0-1/{h_0}}^{h_0}\!\!
\big[\tilde J(h_0-y)+\tilde J(h_0+y)\big]h_0(h_0-y)\dy\nonumber\\
&=&d\int_{-h_0}^{-1/{h_0}}\!\!J(y)\dy+d\int_{h_0-1/{h_0}}^{h_0}\!\!J(h_0-y)h_0(h_0-y)\dy.
 \eess
Letting $h_0\to\yy$ leads to $\tilde d=d$. Thus \eqref{5.1} holds for removing $\tilde d$ and $d$.

Moreover, by the equation of free boundary, we have that for $t>0$,
\bess
h'(t)&=&\tilde\mu\int_{-h(t)}^{h(t)}\int_{h(t)}^{\yy}\tilde J(x-y)\tilde{u}(t,x)\dy\dx\\
&=&\tilde\mu\int_0^{h(t)}\!\int_{h(t)}^{\yy}\big[\tilde J(x-y)+\tilde J(x+y)\big]u(t,x)\dy\dx\\
&=&\mu\int_0^{h(t)}\!\int_{h(t)}^{\yy}J(x-y)u(t,x)\dy\dx.
\eess
By continuity, we deduce
\[\tilde\mu\int_0^{h_0}\int_{h_0}^{\yy}\big[\tilde J(x-y)+\tilde J(x+y)\big]u_0(x)\dy\dx=\mu\int_0^{h_0}\int_{h_0}^{\yy}J(x-y)u_0(x)\dy\dx\]
is valid for all $h_0>0$. This implies that there exists a $x_0\in[0,h_0]$ such that
  \[\tilde\mu\int_{h_0}^{\yy}\big[\tilde J(x_0-y)+\tilde J(x_0+y)\big]\dy
  =\mu\int_{h_0}^{\yy}J(x_0-y)\dy.\]
Then setting $h_0\to0$ gives $2\tilde\mu=\mu$. Therefore,
\bess\int_0^{h_0}\!\int_{h_0}^{\yy}\big[\tilde J(x-y)+\tilde J(x+y)\big]u_0(x)\dy\dx=2\int_0^{h_0}\!\int_{h_0}^{\yy}J(x-y)u_0(x)\dy\dx\eess
holds for all $h_0>0$. Set
\[\Phi(h_0)=\int_0^{h_0}\!\int_{h_0}^{\yy}\big[\tilde J(x-y)+\tilde J(x+y)\big]u_0(x)\dy\dx-2\int_0^{h_0}\!\int_{h_0}^{\yy}J(x-y)u_0(x)\dy\dx\]
for all $h_0>0$. By $\Phi(h_0)\equiv0$, we see $\Phi'(h_0)\equiv0$ in $h_0>0$. Note that $u_0(h_0)=0$. It is easy to show
\[\Phi'(h_0)=-\int_0^{h_0}\big[\tilde J(x-h_0)+\tilde J(x+h_0)-2J(x-h_0)\big]u_0(x)\dx\equiv0, ~ \forall ~ h_0>0,\]
which, together with \eqref{5.1}, yields
\[2\int_0^{h_0}J(x-h_0)u_0(x)\dx=\int_0^{h_0}J(x-h_0)u_0(x)\dx.\]
This is a contradiction since $\int_0^{h_0}J(x-h_0)u_0(x)\dx>0$. The proof is complete.
\end{proof}

\end{document}